\documentclass[final]{siamltex}

\usepackage{algorithm,algorithmic} 
\usepackage{amsfonts, amsmath,amssymb}
\usepackage{graphicx, subfig}
\usepackage[rightcaption]{sidecap} 
\usepackage{notes2bib}
\usepackage[normalem]{ulem}

\usepackage[pdftex]{hyperref}
\hypersetup{colorlinks,citecolor=black,
  filecolor=black,linkcolor=black,urlcolor=black}

\newcommand{\comment}[1]{\vspace{5
    mm}\par \noindent \framebox{\begin{minipage}[c]{0.9 \textwidth}
      \tt COMMENT: #1 \end{minipage}}\vspace{5 mm}\par}
\renewcommand{\comment}[1]{}

\newtheorem{remark}[theorem]{Remark}

\title{Analysis of a method to parameterize planar curves immersed in
  triangulations}

\author{Ramsharan Rangarajan\thanks{Supported by Stanford Graduate
    Fellowship, Stanford University. (email: {\tt
      rram@alumni.stanford.edu})} \and Adrian
  J. Lew\thanks{Corresponding author. Supported by ONR Young
    Investigator Award N000140810852, NSF Career Award CMMI-0747089,
    Department of the Army Research Grant W911NF-07-2-0027. (email:
    {\tt lewa@stanford.edu}).  Department of Mechanical Engineering,
    Stanford University}}

\begin{document}
\maketitle

\begin{abstract} 
  We prove that a planar $C^2$-regular boundary $\Gamma$ can always be
  parameterized with its closest point projection $\pi$ over a certain
  collection of edges $\Gamma_h$ in an ambient triangulation, by
  making simple assumptions on the background mesh. For $\Gamma_h$, we
  select the edges that have both vertices on one side of $\Gamma$ and
  belong to a triangle that has a vertex on the other side. By
  imposing restrictions on the size of triangles near the curve and by
  requesting that certain angles in the mesh be strictly acute, we
  prove that $\pi:\Gamma_h\rightarrow\Gamma$ is a homeomorphism, that
  it is $C^1$ on each edge in $\Gamma_h$ and provide bounds for the
  Jacobian of the parameterization. The assumptions on the background
  mesh are both easy to satisfy in practice and conveniently verified
  in computer implementations. The parameterization analyzed here was
  previously proposed by the authors and applied to the construction
  of high-order curved finite elements on a class of planar piecewise
  $C^2$-curves.
\end{abstract}



\pagestyle{myheadings} \thispagestyle{plain} \markboth{Rangarajan,
  R. and Lew, A.J.}{Parameterization of planar curves}

\section{Introduction}
\label{sec:introduction} 
\begin{figure}
  \centering
  \includegraphics[width=\textwidth]{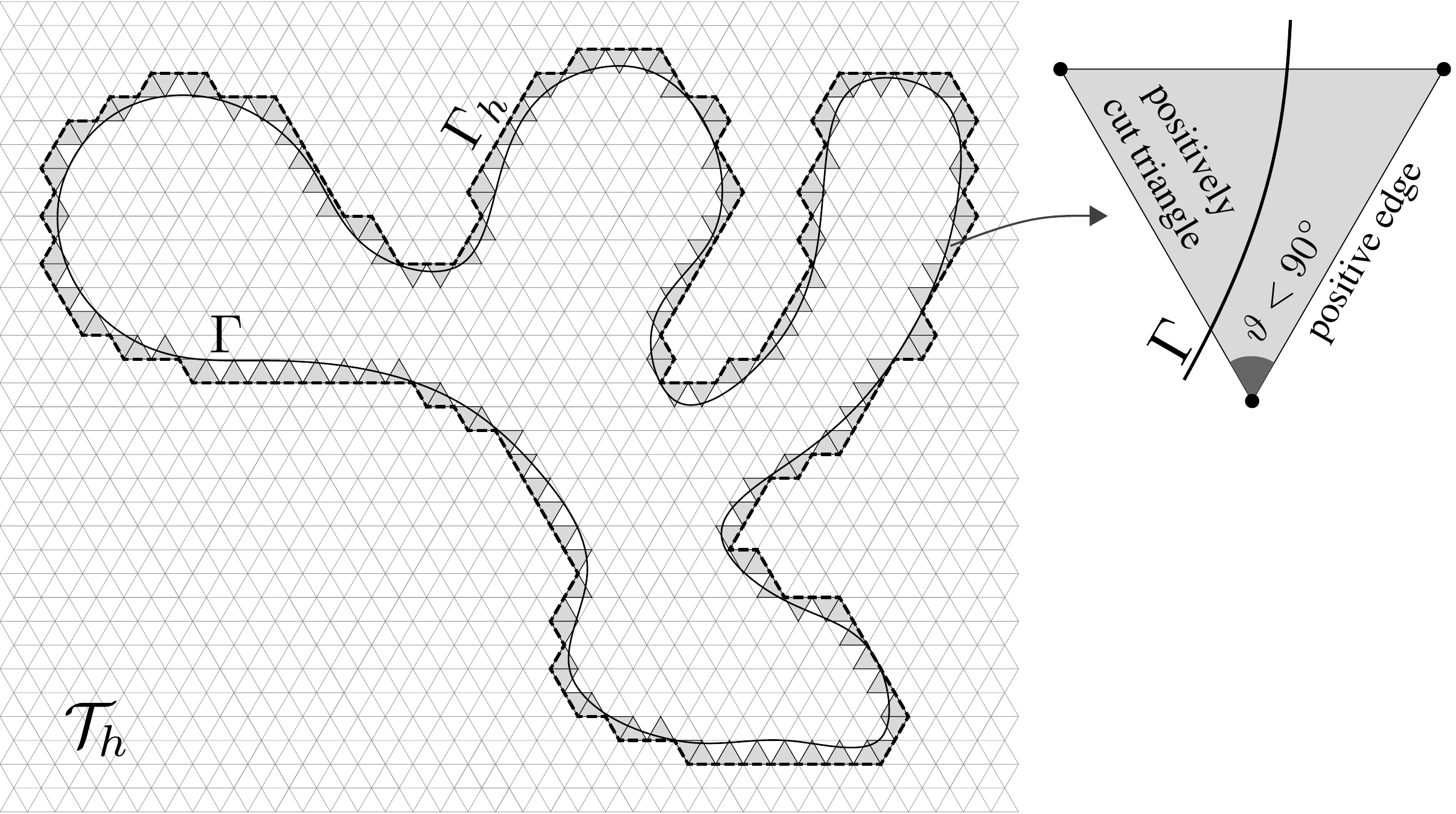}
  \caption{Illustration of the choice of edges in an ambient
    triangulation used to parameterize a $C^2$-regular boundary. The
    curve $\Gamma$ is a cubic spline and is immersed in a
    nonconforming mesh of equilateral triangles. Triangles having one
    vertex inside the region enclosed by $\Gamma$ and two vertices
    outside are said to be positively cut and are shaded in gray. The
    edge of such a triangle joining its two vertices outside is called
    a positive edge; their union is denoted by $\Gamma_h$ and is drawn
    in dotted black lines. Theorem \ref{thm:parameterization}
    identifies sufficient conditions for $\pi:\Gamma_h\rightarrow
    \Gamma$ to constitute a parameterization for $\Gamma$, where $\pi$
    is its closest point projection. A critical one among these
    conditions is that a specific angle in each positively cut
    triangle be strictly acute, namely the one at the vertex of the
    positive edge closest to $\Gamma$, as illustrated in the triangle
    on the right.}
  \label{fig:1}
\end{figure}
The purpose of this article is to analyze a method to parameterize
planar $C^2$-regular boundaries over a collection of edges in a
background triangulation. Such a parameterization was introduced by
the authors in \cite{rangarajan2011parameterization}. The method
consists in making specific choices for the edges in the background
mesh and for the map from these edges onto the curve. For the edges,
we select the ones that have both vertices on one side of the
(orientable) curve to be parameterized and belong to a triangle that
has a vertex on the other side, as illustrated in Fig. \ref{fig:1}.
Such edges are termed positive edges.  For the map, we select the
closest point projection of the curve. In this article, we prove that
the closest point projection restricted to the collection of positive
edges is a homeomorphism onto the curve and that it is $C^1$ on each
positive edge (Theorem \ref{thm:parameterization}).  For this, we have
to impose restrictions on the size of a few triangles near the curve
and request that certain angles in the background mesh be strictly
smaller than $90^\circ$. We also compute bounds for the Jacobian of
the resulting parameterization for the curve.

It is perhaps common knowledge that a sufficiently smooth curve can be
parameterized with its closest point projection over the collection of
interpolating edges in an adequately refined conforming triangulation.
With Theorem \ref{thm:parameterization}, we generalize such an
intuitive parameterization to also include nonconforming background
meshes. In place of the interpolating edges in a conforming mesh, we
pick the collection of positive edges in a nonconforming one, while
still adopting the closest point projection to parameterize the
curve. However, regularity for the curve and refinement for the mesh
do not suffice. We also require certain angles in the mesh to be
strictly acute, as depicted in Fig. \ref{fig:1}. In practice, such an
assumption is both easy to check and satisfy. It is perhaps surprising
that a local algebraic condition on angles in triangles near the curve
precipitates a global topological result. More so, because the angles
required to be acute are irrelevant in the parameterization itself---
neither the identification of positive edges nor the mapping onto the
curve (the closest point projection) depend on them.

A compelling consequence of Theorem \ref{thm:parameterization} is that
\emph{any} planar smooth boundary can be parameterized with its
closest point projection over the collection of positive edges in
\emph{any} sufficiently refined background mesh of equilateral
triangles. It is also interesting to note that the theorem does not
guarantee the same with a background mesh of right-angled
triangles. Such meshes may not satisfy the required assumption on
angles, see \eqref{eq:cond-angle} in Theorem
\ref{thm:parameterization}. On a related note, in
\cite{rangarajan2012universal, rangarajan2011parameterization} we
describe a way of parameterizing curves over edges and diagonals of
meshes of parallelograms, which in particular includes structured
meshes of rectangles. See also \cite{borgers1990triangulation} for a
triangulation algorithm with a similar objective.

The parameterization studied is independent of the particular
description adopted for the curve, is easy to implement and readily
parallelizable.  It also extends naturally to planar curves with
endpoints, corners, self-intersections, T-junctions and practically
all planar curves of interest in engineering and computer graphics
applications, see \cite{rangarajan2011parameterization} and
\cite[Chapter 4]{rangarajan2012universal}. The idea is to construct
such curves by splicing arcs of $C^2$-regular boundaries and
parameterize each arc with its closest point projection.

One of the main motivations behind the parameterization over positive
edges is to accurately represent planar curved domains over
nonconforming background meshes. For once the curved boundary is
parameterized over a collection of nearby edges, we show in
\cite[Chapter 5]{rangarajan2012universal} how a suitable collection of
triangles in the background mesh can be mapped to curved ones to yield
an exact spatial discretization for the curved domain. The
construction of such mappings from straight triangles to curved ones
and their analysis in the context of high-order finite elements with
optimal convergence properties has been the subject of numerous
articles; we refer to a representative few
\cite{ciarlet1972interpolation, ergatoudis1968curved,
  gordon1973transfinite, lenoir1986optimal,
  mansfield1978approximation, scott1973finite, zlamal1973finite} for
details on this subject. Almost without exception, these constructions
have two assumptions in common: (i) a mesh with edges that interpolate
the curved boundary and (ii) a (local) parametric representation for
the curved boundary. The former entails careful mesh generation while
the latter is a strong assumption on how the boundary is described.
The parameterization analyzed here enables relaxing both these
assumptions. 



An outline of the proof of Theorem \ref{thm:parameterization} is given
in \S\ref{subsec:outline}. The crux of the proof is demonstrating
injectivity of the closest point projection $(\pi)$ over the
collection of positive edges $(\Gamma_h)$. Regularity of the
parameterization and estimates for the Jacobian follow easily from
regularity of the curve $(\Gamma)$ and some straightforward
calculations. We prove injectivity by inspecting the restriction of
$\pi$ to each positive edge, then to pairs of intersecting positive
edges, and finally to connected components of $\Gamma_h$. That certain
angles in the mesh be acute has a simple geometric motivation (see
Fig. \ref{fig:line-idea}) and ensures injectivity over each positive
edge (\S\ref{sec:distances-angles},
\S\ref{sec:edge-injective}). Extending this to the entire set
$\Gamma_h$ is non-trivial, requiring some careful, albeit simple
topological arguments. It entails understanding how and how many
positive edges intersect at each vertex in $\Gamma_h$, leading us to
show in \S\ref{sec:on-gammah} that each connected component of
$\Gamma_h$ is a Jordan curve. We then show in \S\ref{sec:pi-injective}
that the restriction of $\pi$ to each connected component of
$\Gamma_h$ is a parameterization of a connected component of
$\Gamma$. Finally in \S\ref{sec:connected}, we establish a
correspondence between connected components of $\Gamma$ and
$\Gamma_h$.


\section{Preliminary definitions}
\label{sec:preliminaries}
In order to state our main result with the requisite assumptions, a
few definitions are essential. First, we define the family of planar
$C^2$-regular boundaries, the curves we consider for parameterization.
\begin{definition}[{\cite[def. 1.2]{henry2005perturbation}}]
  \label{def:Ck-boundary}
  A bounded open set $\Omega\subset{\mathbb R}^2$ has a $C^2$-regular
  boundary if there exists $\Psi\in C^2\left({\mathbb R}^2,{\mathbb
      R}\right)$ such that $\Omega = \left\{ x\in{\mathbb
      R}^2:\Psi(x)<0 \right\}$ and $\Psi(x) = 0$ implies
  $\left|\nabla\Psi\right|\geq1$.  We say that $\Omega$ is a
  $C^2$-regular domain and that $\partial\Omega$ is a $C^2$-regular
  boundary. The function $\Psi$ is called a defining function for
  $\Omega$.
\end{definition}

There are a few equivalent notions of $C^2$-regular boundaries (and
more generally $C^k$-regular boundaries), see
\cite{krantz1999geometry}.  For future reference, we note that each
connected component of a $C^2$-regular boundary is a Jordan curve with
bounded curvature.

We recall the definitions of the signed distance function and the
closest point projection for a curve $\Gamma$ that is the boundary of
an open and bounded set $\Omega$ in ${\mathbb R}^2$. The signed
distance to $\Gamma$ is the map $\phi:{\mathbb R}^2\rightarrow
{\mathbb R}$ defined as $-\min_{y\in\Gamma}d(\cdot ,y)$ over $\Omega$
and as $\min_{y\in\Gamma}d(\cdot ,y)$ elsewhere.  The function
$d(\cdot,\cdot)$ is the Euclidean distance in ${\mathbb R}^2$.  The
closest point projection $\pi$ onto $\Gamma$ is the map $\pi:{\mathbb
  R}^2\rightarrow \Gamma$ given by $\pi(\cdot ) =
\arg\min_{y\in\Gamma}d(\cdot ,y)$.

The following theorem quoted from \cite{henry2005perturbation} is a
vital result for our analysis.  It concerns the regularity of the maps
$\phi$ and $\pi$ for a $C^2$-regular boundary.  The theorem also shows
that $\phi$ is a defining function for a $C^2$-regular domain. In the
statement, the $\varepsilon$-ball centered at $x\in{\mathbb R}^2$ is
the set $B(x,\varepsilon):=\{y:d(x,y)<\varepsilon\}$ and the
$\varepsilon$-neighborhood of $A\subset{\mathbb R}^2$ is the set $B(A,
\varepsilon):=\cup_{x\in A}B(x,\varepsilon)$.
\begin{theorem}[{\cite[Theorem 1.5]{henry2005perturbation}}] If
  $\Omega\subset{\mathbb R}^2$ is an open set with a $C^2$-regular
  boundary, then there exists $r_n>0$ such that
  $\phi:B(\partial\Omega,r_n)\rightarrow (-r_n,r_n)$ and
  $\pi:B(\partial\Omega,r_n)\rightarrow \partial\Omega$ are well
  defined. The map $\phi$ is $C^2$ while $\pi$ is a $C^1$ retraction
  onto $\partial\Omega$. The mapping
  $x\mapsto(\phi(x),\pi(x)):B(\partial\Omega,r_n)\rightarrow
  (-r_n,r_n)\times\partial\Omega$ is a $C^1$-diffeomorphism with
  inverse $(\phi,\xi)\mapsto\xi+\phi\hat{N}(\xi):
  (-r_n,r_n)\times\partial\Omega\rightarrow B(\partial\Omega,r_n)$
  where $\hat{N}(\xi)$ is the unit outward normal to $\partial\Omega$
  at $\xi$. Furthermore, $\phi$ is the unique solution of
  $|\nabla\phi|=1$ in $B(\partial\Omega,r_n)$ with $\phi=0$ on
  $\partial\Omega$ and $\nabla \phi \cdot \hat{N}>0$ on
  $\partial\Omega$.
  \label{thm:nbd}
\end{theorem}

\noindent In Theorem \ref{thm:nbd}, by saying that $\phi$ and $\pi$
are well defined over $B(\partial \Omega, r_n)$, we mean that these
maps are defined and have a unique value at each point in $B(\partial
\Omega, r_n)$. The following proposition follows from \cite[\S
14.6]{gilbarg2001elliptic}. A simple derivation specific to planar
curves can be found in \cite{rangarajan2011parameterization}.
\begin{proposition}
  Let $\Gamma\subset {\mathbb R}^2$ be a $C^2$-regular boundary with
  signed distance function $\phi$, closest point projection $\pi$,
  signed curvature $\kappa_s$, and unit tangent $\hat{T}$. If $p\in
  B(\Gamma, r_n)$ and $|\phi(p)\kappa_s(\pi(p))|<1$, then
  \begin{subequations}
    \label{eq:gradpi-gradgradphi-estimate}
    \begin{align}
      \nabla \pi(p) &= \frac{\hat{T}(\pi(p))\otimes\hat{T}(\pi(p))}
      {1-\phi(p)\,\kappa_s(\pi(p))}, 
      \label{eq:gradpi} \\
      \text{and} \quad \nabla \nabla \phi(p) &=
      -\kappa_s(\pi(p))\nabla \pi(p). \label{eq:gradgradphi}
    \end{align}
  \end{subequations}
  \label{prop:gradpi}
\end{proposition}

For parameterizing $C^2$-regular boundaries, we will consider
background meshes that are triangulations of polygonal domains
(cf. \cite[Chapter 4]{lai2007spline}). We mention the related
terminology and notation used in the remainder of the article. With
triangulation ${\cal T}_h$, we associate a pairing $(V,C)$ of a vertex
list $V$ that is a finite set of points in ${\mathbb R}^2$ and a
connectivity table $C$ that is a collection of ordered $3$-tuples in
$V\times V\times V$ modulo permutations. A vertex in ${\cal T}_h$ is
thus an element of $V$ (and hence a point in ${\mathbb R}^2$). An edge
in ${\cal T}_h$ is a closed line segment joining two vertices of a
member of $C$. The relative interior of an edge $e_{pq}$ with
endpoints (or vertices) $p$ and $q$ is the set
$\textbf{ri}\,(e_{pq})=e_{pq}\setminus\{p,q\}$.

A triangle $K$ in ${\cal T}_h$, denoted $K\in{\cal T}_h$, is the
interior of the triangle in ${\mathbb R}^2$ with vertices given by its
connectivity $\hat{K}\in C$. Frequently, we will not distinguish
between $K$ and $\hat{K}$ unless the distinction is essential. We
refer to the diameter of $K$ by $h_K$ and the diameter of the largest
ball contained in $\overline{K}$ by $\rho_K$. The ratio
$\sigma_K:=h_K/\rho_K$ is called the shape parameter of $K$
\cite[Chapter 3]{lai2007spline}. Later, we will invoke the fact that
$\sigma_K\geq \sqrt{3}$ with equality holding for equilateral
triangles.

To consider curves immersed in background triangulations, we introduce
the following terminology.
\begin{definition}
  \label{def:pos-cut} 
  Let $\Gamma\subset {\mathbb R}^2$ be a $C^2$-regular boundary with
  signed distance function $\phi$ and let ${\cal T}_h$ be a
  triangulation of a polygon in ${\mathbb R}^2$.
  \begin{romannum}
  \item We say that $\Gamma$ is immersed in ${\cal T}_h$ if $\Gamma
    \subset \text{int}\left(\,\cup_{K\in{\cal T}_h}\,\overline{K}
      \,\right)$.
  \item A triangle in ${\cal T}_h$ is \textit{positively cut} by
    $\Gamma$ if $\phi\geq 0$ at precisely two of its vertices.
  \item An edge in ${\cal T}_h$ is a positive edge if $\phi\geq 0$ at
    both of its vertices and if it is an edge of a triangle that is
    positively cut by $\Gamma$.
  \item The proximal vertex of a triangle positively cut by $\Gamma$
    is the vertex of its positive edge closest to $\Gamma$. When both
    vertices of the positive edge are equidistant from $\Gamma$, the
    one containing the smaller interior angle is designated to be the
    proximal vertex. If the angles are equal as well, either vertex of
    the positive edge can be assigned the proximal vertex.
  \item The conditioning angle of a triangle positively cut by
    $\Gamma$ is the interior angle at its proximal vertex.
  \item Let $K, K^\text{adj}\in {\cal T}_h$ be such that $K$ is
    positively cut by $\Gamma$, $K$ has positive edge $e$, $e\cap
    \Gamma\neq \emptyset$ and $\overline{K}\cap
    \overline{K^\text{adj}}=e$.  Then, the angle adjacent to the
    positive edge of $K$, denoted $\vartheta_K^\text{adj}$, is defined
    as the minimum of the interior angles in $K^\text{adj}$ at the
    vertices of $e$.
    
  \end{romannum}
\end{definition}

\section{Main result}
\label{sec:main-result} 
The main result of this article is the following.
\begin{theorem}
  Consider a $C^2$-regular boundary $\Gamma\subset {\mathbb R}^2$ with
  signed distance function $\phi$, closest point projection $\pi$ and
  curvature $\kappa$.  Let $\Gamma$ be immersed in a triangulation
  ${\cal T}_h$. Denote the union of positive edges in ${\cal T}_h$ by
  $\Gamma_h$ and the collection of triangles positively cut by
  $\Gamma$ in ${\cal T}_h$ by ${\cal P}_h$. For each $K\in {\cal
    P}_h$, let
  \begin{align*}
    \vartheta_K &:= \text{conditioning angle of $K$}, \\
    \vartheta_K^\text{adj} &:= \text{angle adjacent to positive edge of
      $K$ when defined}, \\
  M_K &:= \max_{\overline{B(K,h_K)}\cap\Gamma} \kappa \quad
  \text{and} \quad
  C_K^h := \frac{M_K}{1-M_Kh_K}.
\end{align*}
Assume that for each connected component $\gamma$ of $\Gamma$,
$\gamma_h:=\{x\in \Gamma_h\,:\,\pi(x)\in\gamma\}\neq \emptyset$.  If
for each $K\in {\cal P}_h$, we have
\begin{subequations}
  \label{eq:mesh-restrictions}
  \begin{align} 
    h_K&<r_n, \label{eq:h-global} \\
    \vartheta_K & < 90^\circ ~\,  \label{eq:cond-angle}\\
    0<\sigma_KC_K^hh_K &< \min\left\{\cos\vartheta_K,
      \sin\frac{\vartheta_K}{2}\right\}, \label{eq:sigmaCh}\\
    \text{and} \quad C_K^hh_K &< \frac{1}{2}\sin\vartheta_K^\text{adj}
    \quad \text{whenever}~\vartheta_K^\text{adj}~\text{is defined,} \label{eq:Ch}
    \end{align}
  \end{subequations}
  then
  \begin{romannum}
  \item each positive edge in $\Gamma_h$ is an edge of precisely one
    triangle in ${\cal P}_h$,
  \item for each positive edge $e\subset \Gamma_h$, $\pi$ is a
    $C^1$-diffeomorphism over $\textbf{ri}\,(e\,)$,
  \item if $K=(p,q,r)\in{\cal P}_h$ has positive edge $e_{pq}$, then
    \begin{align}
      -&C_K^hh_K^2  < \phi(x) \leq h_K \quad  \forall x\in
      e_{pq}. \label{eq:dist-estimate}
    \end{align}
    The Jacobian $J$ of the map $\pi:\textbf{ri}\,(e_{pq}) \rightarrow
    \Gamma$ satisfies
    \begin{align}
      0<\frac{\sin\left(\beta_K-\vartheta_K \right)}{1+M_Kh_K} 
      \leq J(x)&=\left|\nabla\pi(x)\cdot \frac{(p-q)}{d(p,q)}\right|
      \leq \frac{1}{1-M_Kh_K} \quad \forall
      x\in \textbf{ri}\,(e_{pq}),\label{eq:jac-estimate} \\
      \text{where} \quad 
      \cos\beta_K &:= C_K^h\sigma_Kh_K-\eta_K,
      ~\beta_K\in\left[0^\circ,180^\circ\right], \label{eq:def-betaK}\\
      \eta_K &:= \frac{\min\{\phi(p),
        \phi(q)\}-\phi(r)}{h_K}. \label{eq:def-etaK} 
    \end{align}
  \item The map $\pi:\Gamma_h\rightarrow \Gamma$ is a
    homeomorphism. In particular, $\gamma_h$ as defined above is a
    simple, closed curve.
  \end{romannum} 
  \label{thm:parameterization}
\end{theorem}


\subsection{Discussion of the statement}
\label{subsec:discussion-statement}
With $\Gamma$ and $\Gamma_h$ as in the statement, Theorem
\ref{thm:parameterization} asserts sufficient conditions under which
$\pi:\Gamma_h \rightarrow \Gamma$ is a homeomorphism. The statement of
the theorem extends also to the case when edges in $\Gamma_h$ are
identified using the function $-\phi$ instead of $\phi$. This
corresponds to selecting the collection of \textit{negative} edges for
parameterizing $\Gamma$. Of course, a different collection of angles
are required to be acute. If triangles in the vicinity of the curve
are all acute angled, the theorem shows that there are two different
collections of edges homeomorphic to $\Gamma$.

We make two important assumptions on the background mesh; we briefly
examine them and discuss how they can be satisfied in practice in
\S\ref{subsec:h-bound}.  The first assumption is, expectedly, on the
size of triangles near $\Gamma$, as conveyed by conditions
\eqref{eq:h-global}, 
\eqref{eq:sigmaCh} and \eqref{eq:Ch}. For instance, if the mesh size
is too large, then $\pi$ may not even be single valued over
$\Gamma_h$.

Assumption \eqref{eq:cond-angle}, which we term the {\em acute
  conditioning angle assumption}, is perhaps less intuitive. For once
the set $\Gamma_h$ has been identified, the angles that positive edges
make with other edges in the background mesh ${\cal T}_h$ are
irrelevant. Rather, the rationale behind \eqref{eq:cond-angle} is that
it provides a means to control the orientation of positive edges with
respect to local normals to the curve. We explain this idea below
using a simple example.

It is worth emphasizing that the assumptions on the background mesh in
\eqref{eq:mesh-restrictions} are not very restrictive principally
because there is no conformity required with $\Gamma$.  Besides, the
region triangulated by ${\cal T}_h$ can be quite arbitrary and need
only contain $\Gamma$ in the sense of definition
\ref{def:pos-cut}(i). In particular, while considering ambient
triangulations of larger sets, the restrictions on the size, quality
and angles stemming from \eqref{eq:mesh-restrictions} apply only to a
subset of the collection of triangles intersected by $\Gamma$, namely
positively cut triangles and triangles having positive edges that are
intersected by $\Gamma$.

Finally, we mention that Theorem \ref{thm:parameterization} guarantees
a parameterization for $\Gamma$ provided the collection of triangles
positively cut by each of its connected components is non-empty. This
is apparent from the fact that all restrictions on the mesh size and
angles in \eqref{eq:mesh-restrictions} apply only to positively cut
triangles and triangles having positive edges that are intersected by
$\Gamma$.  For instance, if a connected component $\gamma$ of $\Gamma$
is a contained in the interior of a triangle in ${\cal T}_h$, then no
triangle is positively cut by it. Of course, it is possible for the
collection of triangles positively cut by $\gamma$ to be empty in a
multitude of ways.  In principle, sufficient conditions are easily
identified to ensure at least one triangle is positively cut by each
connected component of $\Gamma$.
In practice however, it is much simpler to inspect the sign of $\phi$
at the vertices of triangles and verify the presence of positively cut
triangles rather than check such conditions.

\subsubsection{The acute conditioning angle assumption}
\label{subsubsec:discussion-angle-hypothesis}
\begin{SCfigure}
  \centering 
  \includegraphics[scale=0.8]{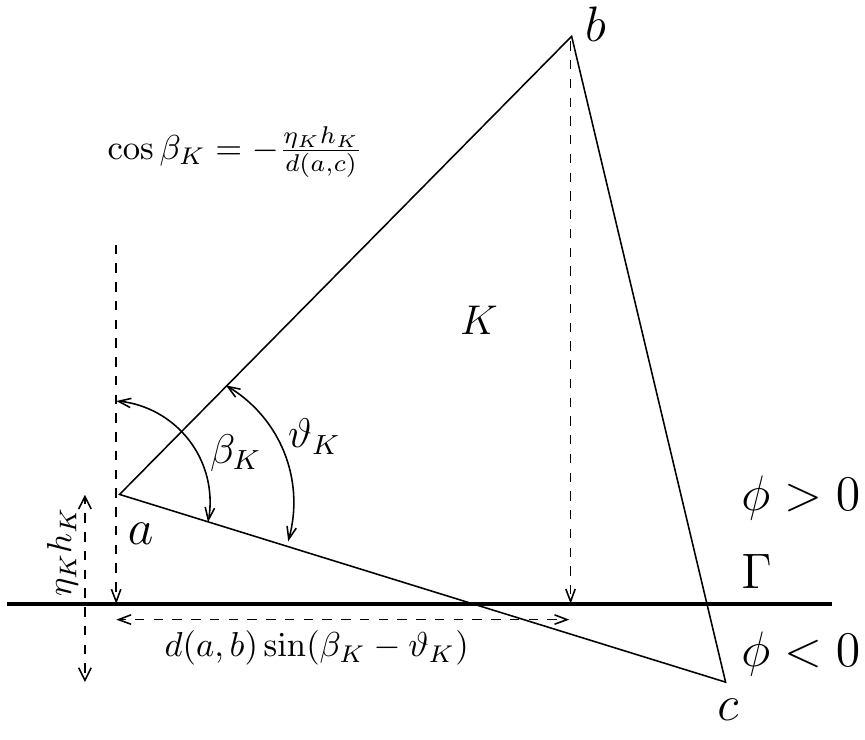}
  \caption{Illustration to explain the rationale behind the acute
    conditioning angle assumption. Triangle $K$ is positively cut by
    $\Gamma$. Although $\beta_K>90^\circ$, it can be arbitrarily close
    to $90^\circ$ by changing the locations of vertices $b$ and
    $c$. Requesting $\vartheta_K<90^\circ$ ensures that
    $\beta_K-\vartheta_K>0^\circ$ always and hence that $\pi(e_{ab})$
    has non-zero length.}
  \label{fig:line-idea}
\end{SCfigure}
Consider a locally straight curve $\Gamma$ as shown in
Fig. \ref{fig:line-idea}. Triangle $K$ shown in the figure is
positively cut by $\Gamma$, has positive edge $e_{ab}$ and proximal
vertex $a$.  Abusing the definition in \eqref{eq:def-betaK}, we have
$\cos\beta_K=-\eta_Kh_K/d(a,c)$ as indicated in the figure (the two
definitions coincide if the length of the edge $e_{ac}$ is $h_K$).
The projection of $e_{ab}$ onto $\Gamma$ has length
$d(a,b)\sin(\beta_K-\vartheta_K)$. For $\pi$ to be injective over
$e_{ab}$, we need to ensure that
$0^\circ<\beta_K-\vartheta_K<180^\circ$.  Even though the angle
$\beta_K$ depicted in the figure is strictly larger than $90^\circ$,
it can be made arbitrarily close to $90^\circ$ by altering the
locations of vertices $a$ and $c$. Therefore, we request that the
conditioning angle $\vartheta_K$ be smaller than $90^\circ$ thereby
ensuring $\beta_K-\vartheta_K>0^\circ$. The assumptions $\phi(a)\leq
\phi(b)$ and $\vartheta_K<90^\circ$ together imply that
$\beta_K-\vartheta_K<180^\circ$.

We refer to \cite{rangarajan2012universal,
  rangarajan2011parameterization} for simple examples where $\pi$
fails to be injective over $\Gamma_h$ because the conditioning angle
fails to be acute.  Of course, \eqref{eq:cond-angle} is only a
sufficient condition for injectivity.  In fact, a simple way to relax
assumption \eqref{eq:cond-angle} is by defining an equivalence
relation $\stackrel{\Gamma}{\simeq}$ over the family of triangulations
in which $\Gamma$ is immersed. Consider two triangulations ${\cal
  T}_h=(V,C)$ and ${\cal T}'_{h'}=(V',C')$. We say ${\cal T}_h
\stackrel{\Gamma}{\simeq} {\cal T}'_{h'}$ if there is a bijection
$\Phi:V\rightarrow V'$ such that
\begin{romannum}
\item $(p,q,r)\in C \iff (\Phi(p),\Phi(q),\Phi(r))\in C'$, 
\item $\phi(v)\geq 0 \iff \phi(\Phi(v))\geq 0$, 
\item $\phi(v)< 0 \iff \phi(\Phi(v))< 0$,
\item $v\in \Gamma_h\Rightarrow \Phi(v)=v$.
\end{romannum}
The map $\Phi$ can be interpreted as a (constrained) perturbation of
vertices in ${\cal T}_h$ to yield a new mesh ${\cal T}'_{h'}$.  It is
clear from the definition of the equivalence relation that both ${\cal
  T}_h$ and ${\cal T}'_{h'}$ have {\em exactly the same set of
  positive edges} even though their positively cut triangles can have
very different conditioning angles. The key point is that the result
of the theorem can be applied to ${\cal T}_h$ from merely knowing the
existence of a triangulation in its equivalence class that has acute
conditioning angle. In light of this observation, the theorem applies
even to some families of background meshes that do not satisfy
assumption \eqref{eq:cond-angle}.

With no conformity requirements on the background mesh, the acute
conditioning angle assumption \eqref{eq:cond-angle} is easy to satisfy
in practice. A simple way for example, is to ensure that triangles in
the vicinity of $\Gamma_h$ in the background mesh are acute angled;
even simpler--- use background meshes consisting of all acute angled
triangles. Such acute triangulations, including adaptively refined
ones, are conveniently constructed by tiling quadtrees using stencils
of acute angled triangles provided in \cite{bern1994provably}.

\subsection{Restrictions on triangle sizes}
\label{subsec:h-bound}
Conditions restricting the mesh size, namely \eqref{eq:h-global},
\eqref{eq:sigmaCh} and \eqref{eq:Ch}, were identified by simply
tracking the restrictions on the mesh size in the proof of Theorem
\ref{thm:parameterization}. They are easily checked for a given curve
and background mesh and can be used to guide refinement of background
meshes near the boundaries of domains. Furthermore, they make
transparent what parameters related to the curve and to the mesh
influence how much refinement is required. For instance,
\eqref{eq:h-global} shows that a more refined mesh is required if the
curve has small features. The requirement that $\sigma_KC_K^hh_K$ be
positive in \eqref{eq:sigmaCh} is equivalent to $M_Kh_K<1$, which
reveals that smaller triangles are required where the curve has large curvature.
More refinement is also needed when conditioning angles are close to
$90^\circ$, when triangles are poorly shaped as indicated by large
values of $\sigma_K$ or small values of $\vartheta_K^\text{adj}$.

Commonly used meshing algorithms usually guarantee \emph{shape
  regularity} and bounds for interior angles in triangles with mesh
refinement. Consequently, there exist mesh size independent constants
$\sigma>0$ and $0^\circ<\theta_\text{min}\leq
\theta_\text{max}<180^\circ$ such that the shape parameter is bounded
by $\sigma$ and interior angles of triangles are bounded between
$\theta_\text{min}$ and $\theta_\text{max}$. As discussed above,
conditioning angles can be guaranteed to be acute independent of the
mesh size. For example, $\vartheta_K= 60^\circ$ for background meshes
of equilateral triangles. Angles in triangulations constructed using
stencils in \cite{bern1994provably} are guaranteed to lie between
$36^\circ$ and $80^\circ$.

It is imperative also to consider if the requirements on the mesh size
posed by \eqref{eq:mesh-restrictions} are too conservative.  We check
this for a specific example of a circle of radius $R$ immersed in a
background mesh of equilateral triangles. In such a case, we have
$h_K=h$ for each triangle in the mesh, and
$\vartheta_K=60^\circ,\sigma_K=\sqrt{3}, r_n=M_K=1/R$ and
$\vartheta_K^\text{adj}=60^\circ$ (when defined) for each positively
cut triangle $K$.  Then, satisfying \eqref{eq:mesh-restrictions}
requires $h<h_0:=R/(1+2\sqrt{3})\simeq 0.224R$. The \textit{a priori}
estimate $h_0=0.224R$ is a reasonable one because it is comparable to
$R$. Of course, the estimate will change with the choice of background
meshes.

\subsubsection{Bound for the Jacobian}
\label{subsubsec:discussion-jacobian}
Eq.\eqref{eq:jac-estimate} provides an estimate for the Jacobian of
the parameterization.  Inspecting the lower bound in
\eqref{eq:jac-estimate}, which is the critical one, shows that $J\geq
\sin(\beta_K-\vartheta_K)$ if $M_K h_K=0$. This is precisely the
Jacobian computed for a line, as in figure \ref{fig:line-idea}, when
the definitions of $\beta_K$ in \eqref{eq:def-betaK} is replaced by
that in the figure. The same interpretation of the lower bound holds
when $M_K\not=0$ but $h_K$ is small. In this case, each positive edge
parameterizes a small subset of $\Gamma$, which appears essentially
straight.

For reasonably large values of $M_Kh_K$, the angle $\beta_K$ in
\eqref{eq:def-betaK} can be close to $90^\circ$, even acute. Hence
$\beta_K-\vartheta_K$ can be small. In light of this, we mention that
a smaller conditioning angle yields a better parameterization, one
with $J$ closer to $1$. Finally, with mesh size independent bounds for
$\sigma_K$ and $\vartheta_K$, it is straightforward to demonstrate
that the estimates for $J$ in \eqref{eq:jac-estimate} are in turn
bounded away from zero independent of the mesh size (specifically,
$\beta_K-\vartheta_K$ and $1\pm M_Kh_K$ appearing in the estimate can
be bounded independent of the mesh size).


\subsection{Outline of proof}
\label{subsec:outline}
We briefly discuss the outline of the proof of Theorem
\ref{thm:parameterization}.  The critical step is showing that $\pi$
is injective over $\Gamma_h$. To this end, we proceed in simple steps
by considering the restriction of $\pi$ over each positive edge, then
over pairs of intersecting positive edges and finally over connected
components of $\Gamma_h$.

In Appendix \ref{sec:distances-angles}, we compute bounds for the
signed distance function $\phi$ on $\Gamma_h$ and for angles between
positive edges and local tangents/normals to $\Gamma$.  By requiring
that size of positively cut triangles be sufficiently small and by
invoking assumption \eqref{eq:cond-angle}, we show that a positive
edge is never parallel to a local normal to $\Gamma$ (Proposition
\ref{prop:angle-normal}). From here, we infer that $\pi$ is injective
over each positive edge (Lemma \ref{lem:edge-injective}). The required
bounds for the Jacobian in \eqref{eq:jac-estimate} also follow easily
from the angle estimates. Part (ii) of the theorem is then a direct
consequence of the inverse function theorem.

A logical next step is to show that $\pi$ is injective over each pair
of intersecting positive edges (Proposition
\ref{lem:pair-injective}). For this, in \S\ref{sec:on-gammah} we first
examine how positive edges intersect.  Lemma \ref{cor:two-vertex}
states that precisely two positive edges intersect at each vertex in
$\Gamma_h$. This result leads us to conclude that $\Gamma_h$ is in
fact a collection of simple, closed curves (Lemma
\ref{lem:simple-closed}).

Knowing that (i) $\pi$ is injective over each pair of intersecting
positive edges, (ii) each connected component of $\Gamma_h$ is a
simple, closed curve and (iii) $\pi$ is continuous over $\Gamma_h$, we
demonstrate (in Lemma \ref{lem:homeo}) that $\pi$ is a homeomorphism
over each connected component of $\Gamma_h$.  What remains to be shown
is that precisely one connected component of $\Gamma_h$ is mapped to
each connected component of $\Gamma$. We do this in
\S\ref{sec:connected} by illustrating that the collection of positive
edges that map to a connected component of $\Gamma$ is itself a
connected set (Lemma \ref{lem:connected}).

\subsection{Assumptions and notation for subsequent sections}
\label{subsec:assumptions}
In all results stated in subsequent sections, we presume that the
\eqref{eq:mesh-restrictions} in the statement of Theorem
\ref{thm:parameterization} hold. In several intermediate results, one
or more of these assumptions could be relaxed.

We shall denote the unit normal and unit tangent to $\Gamma$ at
$\xi\in\Gamma$ by $\hat{N}(\xi)$ and $\hat{T}(\xi)$ respectively. We
assume an orientation for $\Gamma$ such that $\hat{N}=\nabla\phi$ on
the curve, and that $\{\hat{T},\hat{N}\}$ constitutes a right-handed
basis for ${\mathbb R}^2$ at any point on the curve. Given distinct
points $a,b\in{\mathbb R}^2$, we denote the unit vector pointing from
$a$ to $b$ by $\hat{U}_{ab}$ and define $\hat{U}_{ab}^\perp$ such that
$\{\hat{U}_{ab},\hat{U}_{ab}^\perp\}$ is a right-handed basis.

The following simple calculation establishes the ranges of parameters
$\eta_K$ and $\beta_K$ introduced in the statement of Theorem
\ref{thm:parameterization}. Furthermore, for each $K\in {\cal P}_h$,
part (ii) of Proposition \ref{prop:well-def} together with
\eqref{eq:h-global} implies that $\overline{K}\subset
B(\Gamma,r_n)$. Then Theorem \ref{thm:nbd} shows that $\pi$ is $C^1$
and in particular, well defined over $\overline{K}$. Since any
positive edge is an edge of some triangle in ${\cal P}_h$, we get that
$\Gamma_h\subset B(\Gamma,r_n)$ and hence that $\pi$ is well defined
and continuous on $\Gamma_h$. We shall frequently use these
consequences of the proposition in the remainder of the article, often
without explicitly referring to it.
\begin{proposition}
  Let $K=(a,b,c)\in {\cal P}_h$. Then
  \begin{romannum}
  \item $\overline{K}\cap \Gamma\neq \emptyset$. In particular,
    $\phi(c)<0\Rightarrow e_{bc}\cap\Gamma\neq
    \emptyset,e_{ac}\cap\Gamma\neq \emptyset$,
  \item $|\phi|\leq h_K$ on $\overline{K}$,
  \item $\eta_K$ defined in \eqref{eq:def-etaK} satisfies $0<\eta_K\leq
    1$,
  \item $\beta_K$ given by \eqref{eq:def-betaK} is well defined and
    $\beta_K>\vartheta_K$.
  \end{romannum}
  \label{prop:well-def}
\end{proposition}
\begin{proof}
  We only show (iv) and the upper bound in (iii), since the others
  follow directly from the definitions. To this end, assume that
  $\phi(c)<0$ and $\phi(a),\phi(b)\ge 0$, and consider any $\xi\in
  e_{ac}\cap \Gamma$. From the definition of $\eta_K$ in
  (\ref{eq:def-etaK}), we have
  \begin{align}
    \eta_Kh_K\leq \phi(a)-\phi(c)\leq d(a,\Gamma)+d(c,\Gamma)\leq
    d(a,\xi)+d(c,\xi) \leq h_K, \notag
  \end{align}
  which shows that $\eta_K\leq 1$. To show that $\beta_K$ is well
  defined, we check that $\cos\beta_K\in [-1,1]$. Noting that
    $\sigma_KC_K^hh_K\geq 0$ and $\eta_K\leq 1$ shows that 
    \begin{align*}
      \cos\beta_K = \sigma_KC_K^hh_K -\eta_K \geq -\eta_K \geq -1. 
    \end{align*}
    For the upper bound, we have
    \begin{align*}
      \cos\beta_K 
      &= \sigma_K C_Kh_K -\eta_K, \notag \\
      &\leq \sigma_KC_Kh_K, \qquad ~ \left(\text{using}~\eta_K>0\right)   \notag \\
      &\leq \cos\vartheta_K \leq 1, \qquad
      \left(\text{from}~\eqref{eq:sigmaCh}\right) 
    \end{align*}
    which also shows that $\beta_K>\vartheta_K$.
\end{proof}
\comment{
  The proof just involves exercising definitions.
  \begin{romannum}
  \item Inspect the sign of the continuous function $\phi$ at vertices
    of $K$.
  \item For any $z\in \overline{K}$ and $\xi\in e_{ac}\cap \Gamma$,
    $|\phi(z)|=d(z,\Gamma)\leq d(z,\xi)\leq h_K$.
  \item Assume $\phi(c)<0$. Then $\phi(a),\phi(b)\geq 0$ which implies
    \begin{align}
      \eta_K = \frac{\min\{\phi(a),\phi(b)\}-\phi(c)}{h_K} >
      0. \label{eq:a-1}
    \end{align}
    Consider any $\xi\in e_{ac}\cap \Gamma$. From the definition of
    $\eta_K$ in \eqref{eq:a-1}, we have
    \begin{align}
      \eta_Kh_K\leq \phi(a)-\phi(c)\leq d(a,\Gamma)+d(c,\Gamma)\leq
      d(a,\xi)+d(c,\xi) \leq h_K, \notag
    \end{align}
    which shows that $\eta_K\leq 1$.
  \item To show that $\beta_K$ to be well defined, we check that
    $\cos\beta_K\in [-1,1]$ using $\eta_K\leq 1$ and
    \eqref{eq:sigmaCh}:
    \begin{align}
      -1\leq -\eta_K\leq \cos\beta_K=\sigma_KC_K^hh_K-\eta_K\leq
      \sigma_KC_K^hh_K<\mu\cos\vartheta_K. \label{eq:a-2}
    \end{align}
    Since $\mu<1$ independent of $h$, \eqref{eq:a-2} also shows that
    $\beta_K>\vartheta_K$ for any $h$. \qquad \endproof
  \end{romannum}
}

\section{Injectivity on each positive edge}
\label{sec:edge-injective}
To show the injectivity of $\pi$ on each positive edge (Lemma
\ref{lem:edge-injective}) and estimate the Jacobian of this mapping
(Lemma \ref{lem:jac}), we essentially follow the calculation
illustrated in Fig. \ref{fig:line-idea}. In both arguments, we use the
following angle estimate that is proved in Appendix
\ref{sec:distances-angles}.
\begin{proposition}
  Let $K=(a,b,c)\in{\cal P}_h$ have positive edge $e_{ab}$ and
  proximal vertex $a$. Then
  \begin{align}
    -\frac{3}{2}C_K^hh_K &\leq \hat{N}(\pi(x))\cdot\hat{U}_{ab} \leq
    \cos(\beta_K-\vartheta_K) \quad \forall x\in e_{ab}.
    \label{eq:angle-normal} 
  \end{align}
  In particular, $|\hat{N}(\pi(x))\cdot\hat{U}_{ab}|<1$ and
  $|\hat{T}(\pi(x))\cdot\hat{U}_{ab}|>0$.
  \label{prop:angle-normal}
\end{proposition}
\begin{lemma}
  The restriction of $\pi$ to each positive edge in $\Gamma_h$ is
  injective.
  \label{lem:edge-injective}
\end{lemma}
\begin{proof}
  Let $(a,b,c)\in{\cal P}_h$ have positive edge $e_{ab}$ and proximal
  vertex $a$.  We proceed by contradiction. Suppose that $x,y\in
  e_{ab}$ are distinct points such that $\pi(x)=\pi(y)$.  From Theorem
  \ref{thm:nbd} and $\pi(x)=\pi(y)$, we have
  \begin{subequations}
    \label{eq:g12}
    \begin{align}
      x &= \pi(x) + \phi(x)\hat{N}(\pi(x)), \label{eq:g1} \\
      y &= \pi(y) + \phi(y)\hat{N}(\pi(y)) = \pi(x) + \phi(y)\hat{N}(\pi(x)). \label{eq:g2}
    \end{align}
  \end{subequations}
  Noting $x\neq y$ in \eqref{eq:g12} implies that
  $\phi(x)\neq\phi(y)$.  Therefore, subtracting \eqref{eq:g2} from
  \eqref{eq:g1} yields
  \begin{align}
    \hat{N}(\pi(x)) &= \frac{x-y}{\phi(x)-\phi(y)}. \label{eq:g3}
  \end{align}
  By definition of $x,y\in e_{ab}$, $x-y$ is a vector parallel to
  $\hat{U}_{ab}$. Therefore \eqref{eq:g3} in fact shows that
  $|\hat{N}(\pi(x))\cdot\hat{U}_{ab}| =1$, contradicting Proposition
  \ref{prop:angle-normal}.
\end{proof}

Before showing the bounds in \eqref{eq:jac-estimate} for the Jacobian,
we prove Corollary \ref{cor:edge-homeo}, a useful step in showing part
(iv) of Theorem \ref{thm:parameterization}.  As discussed in
\S\ref{subsec:assumptions}, continuity of $\pi$ on each positive edge
follows from part (ii) of Proposition \ref{prop:well-def}.  The
continuity of its inverse is a consequence of Lemma
\ref{lem:edge-injective} and the following result in basic topology,
which we use here and later in \S\ref{sec:pi-injective}.
\begin{theorem}[{\cite[Chapter 3]{armstrong1983basic}}]
  A one-one, onto and continuous function from a compact space to a
  Hausdorff space is a homeomorphism.
  \label{thm:inv-cont}
\end{theorem}
\begin{corollary}[of Lemma \ref{lem:edge-injective}]
  Let $e$ be a positive edge in $\Gamma_h$. Then $\pi:e\rightarrow
  \pi(e)$ is a homeomorphism.
  \label{cor:edge-homeo}
\end{corollary}
\begin{proof}
  \comment{ By definition, edge $e$ is a compact set. From part (ii)
    of Proposition \ref{prop:well-def} and Lemma
    \ref{lem:edge-injective}, we know that $\pi$ is continuous and
    injective on $e$. Being a subset of ${\mathbb R}^2$, $\pi(e)$ is a
    Hausdorff space. The corollary then follows from Theorem
    \ref{thm:inv-cont}.}  From part (ii) of Proposition
  \ref{prop:well-def} and Lemma \ref{lem:edge-injective}, we know that
  $\pi$ is continuous and injective on $e$. The corollary then follows
  from Theorem \ref{thm:inv-cont}.
\end{proof}

\begin{lemma}
  Let $K=(a,b,c)\in {\cal P}_h$ have positive edge $e_{ab}$. Then
  $\pi$ is $C^1$ over $\textbf{ri}\,(e_{ab})$ and
  \begin{align}
    0<\frac{\sin\left(\beta_K-\vartheta_K\right)}{1+M_Kh_K} \leq
    \left|\nabla \pi(x)\cdot \hat{U}_{ab}\right| \leq
    \frac{1}{1-M_Kh_K}\le \frac{5}{3} \quad \forall x\in e_{ab}.  \label{eq:jac-main}
  \end{align}
  \label{lem:jac}
\end{lemma}
\begin{proof}
  From part (ii) of Proposition \ref{prop:well-def} and
  \eqref{eq:h-global}, we know $e_{ab}\subset B(\Gamma, r_n)$. Then
  Theorem \ref{thm:nbd} shows that $\pi$ is $C^1$ over
  $\textbf{ri}\,(e_{ab})$.

  Consider any $x\in \textbf{ri}\,(e_{ab})$. Since $|\phi(x)|\leq h_K$
  (Proposition \ref{prop:well-def}), 
  \begin{align}
    \kappa(\pi(x)) \leq \max_{\overline{B(x,h_K)}\cap \Gamma}\kappa
    \leq \max_{\overline{B(K,h_K)}\cap \Gamma}\kappa =
    M_K. \label{eq:h2}
  \end{align}
  Therefore, $|\phi(x)\kappa(\pi(x))|\leq M_Kh_K$ which is smaller
  than $1$ because of the assumption $\sigma_KC_K^hh_K>0$ in
  \eqref{eq:sigmaCh}. Then from Proposition \ref{prop:gradpi}, we get
  \begin{align}
    J(x) &:= \left|\nabla \pi(x)\cdot \hat{U}_{ab}\right| =
    \frac{|\hat{U}_{ab}\cdot
      \hat{T}(\pi(x))|}{|1-\phi(x)\kappa_s(\pi(x))|}, \label{eq:h1}
  \end{align} 
  where $\kappa_s$ is the signed curvature of $\Gamma$ (and
  $\kappa=|\kappa_s|$).  From $|\phi(x)\kappa_s(\pi(x))|\leq
  M_Kh_K<1$, we get
  \begin{align}
    1-M_Kh_K \leq |1-\phi(x)\kappa_s(\pi(x))| \leq 1+M_Kh_K. \label{eq:h3}
  \end{align}
  From Proposition \ref{prop:angle-normal}, we have
  \begin{align}
    \left|\sin(\beta_K-\vartheta_K)\right|\leq
    \left|\hat{T}(\pi(x))\cdot \hat{U}_{ab}\right| \leq 1. \label{eq:h4}
  \end{align}
  Note however from part (iv) of Proposition \ref{prop:well-def} that
  $\beta_K>\vartheta_K \Rightarrow
  \left|\sin(\beta_K-\vartheta_K)\right|=\sin(\beta_K-\vartheta_K)$.
  Then using \eqref{eq:h3} and \eqref{eq:h4} in \eqref{eq:h1} yields
  the  lower and upper bounds for $|J(x)|$ in \eqref{eq:jac-main}.
  
  It remains to show that these bounds are meaningful, i.e., the lower
  bound is positive and the upper bound is not arbitrarily large. The
  former is a consequence of $\beta_K>\vartheta_K$ (from Proposition
  \ref{prop:well-def}). We know from \eqref{eq:sigmaCh} that
  $\sigma_KC_K^hh_K<\sin(\vartheta_K/2)<1$. Then, using $M_Kh_K<1$
  from \eqref{eq:sigmaCh} and $\sigma_K\geq \sqrt{3}$, we get $M_Kh_K <
  (1+\sqrt{3})^{-1}<2/5$, which renders the upper bound in
  \eqref{eq:jac-main} independent of $h_K$.
\end{proof}

By using the strictly positive lower bound for the Jacobian computed
in the inverse function theorem, we conclude that
$\pi:\textbf{ri}(e)\rightarrow \Gamma$ is a locally a
$C^1$-diffeomorphism on each positive edge $e$. Since we have already
shown the injectivity of this map in Lemma \ref{lem:edge-injective},
part (ii) of Theorem \ref{thm:parameterization} follows.

\section{The set $\Gamma_h$}
\label{sec:on-gammah}
An essential step in showing that $\pi$ is injective over $\Gamma_h$
is understanding how positive edges intersect. The goal of this
section is to demonstrate that $\Gamma_h$ is a union of simple, closed
curves (Lemma \ref{lem:simple-closed}). We achieve this by considering
how many positive edges intersect at each vertex in $\Gamma_h$. In
Lemma \ref{cor:two-vertex}, we state that this number is precisely
two. Additionally, as claimed in part (i) of Theorem
\ref{thm:parameterization} and stated below in Lemma
\ref{lem:uniq-id}, each positive edge belongs to precisely one
positively cut triangle. The proofs of these two lemmas is somewhat
laborious, and hence are included in Appendix \ref{sec:topology}.



\begin{lemma}
  Each positive edge in ${\cal T}_h$ is a positive edge of precisely
  one triangle positively cut by $\Gamma$.
  \label{lem:uniq-id}
\end{lemma}

\begin{lemma}
  Precisely two distinct positive edges intersect at each vertex in
  $\Gamma_h$.
  \label{cor:two-vertex}
\end{lemma}

\begin{lemma}
  Let $\gamma_h$ be a connected component of $\Gamma_h$. Then
  $\gamma_h$ is a simple, closed curve that can be represented as as
  \begin{align}
    \gamma_h &= \bigcup_{i=0}^ne_{v_iv_{(i+1)\text{mod}\,(n)}}, \label{eq:simple-closed}
  \end{align}
  where $v_0,\ldots,v_n$ are all the distinct vertices in $\gamma_h$
  and $2\leq n <\infty$.
  \label{lem:simple-closed}
\end{lemma}
\begin{proof}
  We will only prove \eqref{eq:simple-closed}. That $\gamma_h$ is a
  simple and closed curve follows immediately from such a
  representation.

  Denote the number of vertices in $\gamma_h$ by $n+1$ for some integer $n$.
  Since $\gamma_h$ is non-empty, it contains at least one positive edge,
  say $e_{v_0v_1}$ with vertices $v_0$ and $v_1$.  Lemma
  \ref{cor:two-vertex} shows that precisely two positive edges
  intersect at $v_1$. Therefore, we can find vertex $v_2\in
  \gamma_h$ different from $v_0,v_1$ such that $e_{v_1v_2}$ a
  positive edge. This shows that $n\geq 2$. Of course $n<\infty$
  because there are only finitely many vertices in ${\cal T}_h$.
  
  We have identified vertices $v_0,v_1$ and $v_2$ such that
  $e_{v_0v_1},e_{v_1v_2}\subset \gamma_h$.  Suppose that we have
  identified vertices $v_0,v_1,\ldots,v_{k-1}$ for
  $k\in\{2,\ldots,n\}$ such that $e_{v_iv_{(i+1)}}\subset \gamma_h$
  for each $0\leq i\leq k-2$.  We show how to identify vertex $v_k$
  such that $e_{v_{(k-1)}v_k}\subset \gamma_h$.  Lemma
  \ref{cor:two-vertex} shows that precisely two positive edges
  intersect at $v_{(k-1)}$. One of them is
  $e_{v_{(k-2)}v_{(k-1)}}$. Let $v_k$ be such that $e_{v_{(k-1)}v_k}$
  is the other positive edge.  While $v_k$ is different from
  $v_{(k-2)}$ and $v_{(k-1)}$ by definition, it remains to be shown
  that $v_k\neq v_i$ for $0\leq i< k-2$. To this end, note that for
  $1\leq i< k-2$, we have already found two positive edges that
  intersect at $v_i$, namely $e_{v_{(i-1)}v_i}$ and
  $e_{v_iv_{(i+1)}}$. Therefore, it follows from Lemma
  \ref{cor:two-vertex} that $e_{v_iv_{(k-1)}}$ cannot be a positive
  edge for $1\leq i < k-2$. Hence $v_k\neq v_i$ for $1\leq i <
  k-2$. On the other hand, suppose that $v_k=v_0$. Then
  $e_{v_0v_{(k-1)}}$ and $e_{v_0v_1}$ are the two positive edges
  intersecting at $v_0$. In particular, this implies that for each
  $0\leq i\leq k-1$, we have found the two positive edges that
  intersect at vertex $v_i$.  Noting that $n>k-1$, let $w$ be any
  vertex in $\gamma_h$ different from $v_0,\ldots,v_{(k-1)}$. It
  follows from Lemma \ref{cor:two-vertex} that $e_{v_iw}$ cannot
  be a positive edge for any $0\leq i\leq k-1$. This contradicts the
  assumption that $\gamma_h$ is a connected set. Hence $v_k\neq v_0$.
  
  Repeating the above step, we identify all the distinct vertices
  $v_0,\ldots, v_n$ in $\gamma_h$ such that $e_{v_iv_{(i+1)}}$ is a
  positive edge for $0\leq i< n$. All vertices in $\gamma_h$ can be found this way
  because $\gamma_h$ is connected. It only remains to show that $e_{v_nv_0}\subset \gamma_h$.
  The argument is similar to the one given
  above. Lemma \ref{cor:two-vertex} shows that precisely two
  positive edges intersect at $v_n$. One of them is
  $e_{v_{(n-1)}v_n}$. Since $v_0,\ldots,v_n$ are all the vertices in
  $\gamma_h$, the other edge has to be $e_{v_nv_i}$ for some $0\leq
  i\leq n-2$. However, $e_{v_iv_n}$ cannot be a positive edge for
  $1\leq i< n-1$ since we have already identified $e_{v_{(i-1)}v_i}$
  and $e_{v_iv_{(i+1)}}$ as the two positive edges intersecting at
  $v_i$. Hence we conclude that $e_{v_nv_0}$ is a positive edge of
  $\gamma_h$.
\end{proof}


\section{Injectivity on connected components of $\Gamma_h$}
\label{sec:pi-injective}
The main result of this section is the following lemma.
\begin{lemma}
  Let $\gamma$ and $\gamma_h$ be connected components of $\Gamma$ and
  $\Gamma_h$ respectively, such that $\gamma\cap \pi(\gamma_h)\neq
  \emptyset$. Then $\pi:\gamma_h\rightarrow \gamma$ is a
  homeomorphism.
  \label{lem:homeo}
\end{lemma}

Surjectivity of $\pi:\gamma_h\rightarrow \gamma$ in the above lemma
is simple. Continuity of $\pi$ over the connected set $\gamma_h$
implies that $\pi(\gamma_h)$ is a connected subset of $\Gamma$. Since
$\gamma$ is a connected component of $\Gamma$ and $\gamma\cap
\pi(\gamma_h)\neq \emptyset$, $\pi(\gamma_h)\subseteq \gamma$. We also
know that $\pi(\gamma_h)$ is a closed curve because $\gamma_h$ is a
closed curve (Lemma \ref{lem:simple-closed}). Since $\gamma$ is a
Jordan curve, the only closed and connected curve contained in
$\gamma$ is either a point in $\gamma$ or $\gamma$ itself. In view of
Lemma \ref{lem:edge-injective},
$\pi(\gamma_h)$ is not a point, and hence $\pi(\gamma_h)=\gamma$

The critical step is proving injectivity. For this, we extend the result of Lemma
\ref{lem:edge-injective} in Proposition \ref{lem:pair-injective} to
show that $\pi$ is injective over any two intersecting positive edges
in $\gamma_h$ (or $\Gamma_h$). This result does not suffice for an
argument to prove injectivity by considering distinct points in
$\gamma_h$ whose images in $\gamma$ coincide and then arrive a
contradiction.  Instead, we consider a subdivision of $\gamma_h$ into finitely many connected subsets. For a specific choice of these subsets, we demonstrate using Proposition \ref{lem:pair-injective} that $\pi$ is injective over each of these subsets (Proposition \ref{prop:pv}). Then we argue that there can be only one such subset and that it has to equal $\gamma_h$ itself (Proposition \ref{prop:oneloop}). 

\begin{proposition}
  If $e_{ap}$ and $e_{aq}$ are distinct positive edges in $\Gamma_h$,
  then $\pi:e_{ap}\cup e_{aq}\rightarrow \Gamma$ is injective.
  \label{lem:pair-injective}
\end{proposition}
\begin{proof}
  Let $\alpha_i=\arccos(\hat{N}(\pi(a))\cdot\hat{U}_{ai})$ for
  $i=p,q$.  By Lemma \ref{lem:simple}, we know that
  $\hat{T}(\pi(a))\cdot\hat{U}_{ap}$ and
  $\hat{T}(\pi(a))\cdot\hat{U}_{aq}$ have opposite (non-zero) signs.
  Therefore, without loss of generality, assume that
  $\hat{T}(\pi(a))\cdot\hat{U}_{ap}<0$ and
  $\hat{T}(\pi(a))\cdot\hat{U}_{aq}>0$ so that
  \begin{subequations}
    \label{eq:q12}
    \begin{align}
      \hat{U}_{ap} &= \cos\alpha_p\,\hat{N}(\pi(a)) - \sin\alpha_p\,\hat{T}(\pi(a)), \label{eq:q1}\\
      \hat{U}_{aq} &= \cos\alpha_q\,\hat{N}(\pi(a)) +
      \sin\alpha_q\,\hat{T}(\pi(a)). \label{eq:q2}
    \end{align}
  \end{subequations}
  
  We proceed by contradiction. Suppose that $x$ and $y$ are distinct
  points in $e_{ap}\cup e_{aq}$ such that $\pi(x)=\pi(y)$. By Lemma
  \ref{lem:edge-injective}, we know that $\pi$ is injective over
  $e_{ap}$ and $e_{aq}$ respectively. Therefore, $x$ and $y$ cannot
  both belong to either $e_{ap}$ or $e_{aq}$.  Without loss of
  generality, assume that $x\in e_{ap}\setminus\{a\}$ and $y\in
  e_{aq}\setminus\{a\}$. In the following, we identify a point $z\in
  B(\Gamma, r_n)$ such that $\pi(z)$ equals both $\pi(x)$ and
  $\pi(a)$. This will contradict Lemma \ref{lem:edge-injective}.

  Let $0<\lambda_x\leq d(a,p)$ and $0<\lambda_y\leq d(a,q)$ be such
  that
  \begin{subequations}
    \label{eq:q34}
    \begin{align}
      x &= a + \lambda_x\hat{U}_{ap}, \label{eq:q3}\\
      \text{and} \quad 
      y &= a + \lambda_y\hat{U}_{aq}. \label{eq:q4}
    \end{align}
  \end{subequations}
  Consider the point
  \begin{align}
    z &= \pi(x) + \xi\hat{N}(\pi(x)), \label{eq:q5}\\
    \text{where}~ ~ \xi &=
    \frac{\phi(y)\lambda_x\sin\alpha_p+\phi(x)\lambda_y\sin\alpha_q}{\lambda_x\sin\alpha_p+\lambda_y\sin\alpha_q}. \label{eq:q6}
  \end{align}
  Since $\lambda_x,\lambda_y$ are strictly positive (by definition)
  and $\sin\alpha_p,\sin\alpha_q$ are strictly positive (Proposition
  \ref{prop:angle-normal}), we know that
  $\lambda_x\sin\alpha_p+\lambda_y\sin\alpha_q\neq 0$. Hence $z$ given
  by \eqref{eq:q5} is well defined. Moreover, from $|\phi(x)|\leq h_K$
  and $|\phi(y)|\leq h_K$ (Proposition \ref{prop:well-def}), it
  follows from \eqref{eq:q6} that $|\xi|\leq h_K$. Since $h_K<r_n$ by
  \eqref{eq:h-global}, $z\in B(\Gamma, r_n)$. Therefore from
  \eqref{eq:q5} and Theorem \ref{thm:nbd}, we conclude that
  $\pi(z)=\pi(x)$.

  Next we show that $\pi(z)=\pi(a)$ as well. From Theorem
  \ref{thm:nbd} and the assumption that $\pi(y)=\pi(x)$, we have
  \begin{subequations}
    \label{eq:q78}
    \begin{align}
      x &= \pi(x) + \phi(x)\hat{N}(\pi(x)). \label{eq:q7}\\
      \text{and} \quad y &= \pi(y) + \phi(y)\hat{N}(\pi(y)) = \pi(x) +
      \phi(y)\hat{N}(\pi(x)). \label{eq:q8}
    \end{align}
  \end{subequations}
  Observe from \eqref{eq:q78} that $x\neq y \Rightarrow
  \phi(x)\neq\phi(y)$. Hence, subtracting \eqref{eq:q8} from
  \eqref{eq:q7} and using \eqref{eq:q34} yields
  \begin{align}
    \hat{N}(\pi(x)) &= \frac{x-y}{\phi(x)-\phi(y)}=
    \frac{\lambda_x\hat{U}_{ap}-\lambda_y\hat{U}_{aq}}{\phi(x)-\phi(y)}. \label{eq:q9}
  \end{align}
  From \eqref{eq:q3}, \eqref{eq:q5} and \eqref{eq:q7} we get
  \begin{align}
    z =a +
    \lambda_x\hat{U}_{ap}+(\xi-\phi(x))\hat{N}(\pi(x)). \label{eq:q9-a}
  \end{align}
  Upon using \eqref{eq:q12}, \eqref{eq:q6} and \eqref{eq:q9} in
  \eqref{eq:q9-a} and simplifying, we get
  \begin{align}
    z &= a +
    \underbrace{\frac{\lambda_x\lambda_y\sin(\alpha_p+\alpha_q)}{\lambda_x\sin\alpha_p+\lambda_y\sin\alpha_q}}_{\zeta}~\hat{N}(\pi(a))
    = \pi(a) + (\phi(a)+\zeta)\,\hat{N}(\pi(a)). \label{eq:q11}
  \end{align}
  By Theorem \ref{thm:nbd}, \eqref{eq:q11} shows that $\pi(z)=\pi(a)$.
  Hence we have shown that $\pi(x)=\pi(a)$ (both equal point
  $\pi(z)$). This contradicts the fact that $\pi$ is injective on
  $e_{ap}$.
\end{proof}

To proceed, it is convenient to introduce parameterizations for
$\gamma$ and $\gamma_h$. To this end, consider a representation for
$\gamma_h$ as in \eqref{eq:simple-closed}, where $\{v_i\}_{i=0}^n$ are
all of its  vertices. From Lemma \ref{lem:simple-closed} we know that
$\gamma_h$ is a simple, closed curve, so let a
parameterization of $\gamma_h$ be $\alpha:[0,1)\rightarrow
\gamma_h$ continuous and one-to-one such that 
\begin{romannum}
\item $\alpha(0)=\alpha(1^-)=v_0$,
\item $\alpha^{-1}(v_i)<\alpha^{-1}(v_j)$ if $0\leq i<j\leq n$,
\end{romannum}
Clearly
$e_{v_iv_{(i+1)}}=\alpha[\alpha^{-1}(v_i),\alpha^{-1}(v_{i+1})]$ for
$0\leq i<n$ and $e_{v_nv_0}=\alpha[\alpha^{-1}(v_n),1^{-})$.
Similarly, given that $\gamma$ is a simple, closed curve, we consider
a continuous and one-to-one parameterization $\beta:[0,1)\rightarrow \gamma$ of $\gamma$. As discussed at the beginning of this section, the hypotheses in Lemma \ref{lem:homeo} imply that $\pi(\gamma_h)= \gamma$, and in particular that $\pi(v_0)\in \gamma$. Therefore without loss of generality, we assume that $\beta(0)=\beta(1^-)=\pi(v_0)$. For future reference, we note that
$\beta^{-1}:\gamma\setminus \pi(v_0)\rightarrow (0,1)$ is injective
and continuous as well.

We can now define the connected subsets of $\gamma_h$ alluded to at
the beginning of \S \ref{sec:pi-injective}. Let $P_0:=\{p\in [0,1):\pi(\alpha(p))=\pi(v_0)\}$. Observe that since $\pi$ is injective over each positive edge in $\gamma_h$ (Lemma \ref{lem:edge-injective}), each of these edges has at most one point in common with $\alpha(P_0)$. Consequently, $P_0$ is a collection of finitely many points. Then, noting from the definition of $P_0$ that $0\in P_0$, we consider the following ordering for points in $P_0$:
\begin{align}
  P_0 = \{p_i\,:\, 0\le i< m<\infty, \, 0=p_0<p_1<\ldots <p_{m-1}<1\}. \label{eq:def-P0}
\end{align} 
Additionally, for convenience we set $p_m=1$. The connected subsets of $\gamma_h$ we consider are the sets $\alpha([p_i,p_{i+1}))$ for $0\leq i<m$.
\begin{proposition}
  For $0\leq i<m$, $\pi:\alpha[p_i,p_{i+1})\rightarrow \gamma$ is a bijection.
  \label{prop:pv}
\end{proposition}
\begin{proof}
  To prove the proposition, we show that the map
  $\psi:=\beta^{-1}\circ\pi\circ\alpha$ is injective over the interval
  $(p_i,p_{i+1})$. To this end, we will need to consider the
  (positive) edges of $\gamma_h$ contained in $\alpha[p_i,p_{i+1}]$.
  Denote the number of such edges by $k$, set $v_a=\alpha(p_i)$, and
  define $\{q_j\}_{j=0}^{k+1}$ as $q_j=\alpha^{-1}(v_{a+j})$.  Then,
  by the definition of $\alpha$, $\{q_j\}_{j=0}^{k+1}\subset
  [p_i,p_{i+1}]$ and
  \begin{align}
    p_i=q_0<q_1<\ldots <q_{k}<q_{k+1}=p_{i+1}. \label{eq:r0}
  \end{align}
  Notice that $k\geq 1$ because $k=0$ would imply that $\pi$ is not
  injective on the edge containing the points $\alpha(p_i)$ and
  $\alpha(p_{i+1})$, contradicting Lemma \ref{lem:edge-injective}.

  Consider $0\le j\leq k-1$. Proposition \ref{lem:pair-injective} shows
  that $\pi$ is injective over $\alpha[q_j,q_{j+2}]$, and hence $\psi$
  is injective over $(q_j,q_{j+2})$. Since $\psi$ is continuous over
  $(p_i,p_{i+1})$, it is continuous over $(q_j,q_{j+2})$ as
  well. Consequently, $\psi$ is continuous and strictly monotone over
  $(q_j,q_{j+2})$. 

  From here, we conclude that $\psi$ is continuous and strictly monotone over the interval $(q_0,q_{k+1})=(p_i,p_{i+1})$. In particular, $\psi$ is injective over $(p_i,p_{i+1})$. Since $\beta^{-1}$ is injective over $\gamma\setminus \pi(v_0)$, we get that $\pi\circ\alpha$ is injective over $(p_i,p_{i+1})$, i.e., that $\pi$ is injective over $\alpha(p_i,p_{i+1})$. From the definition of $P_0$, we know that $\pi(\alpha(p_i))=\pi(v_0)$ and that $\pi(v_0)\notin\pi(\alpha(p_i,p_{i+1}))$. Therefore we conclude that $\pi$ is in fact injective over $\alpha[p_i,p_{i+1})$. 

  Finally we show $\pi:\alpha[p_i,p_{i+1})\rightarrow \gamma$ is
  surjective.
  Since $\pi$ is continuous over the connected set $\alpha[p_i,p_{i+1})$, $\pi(\alpha[p_i,p_{i+1}))$ is a connected subset of $\gamma$. Since $\pi(\alpha(p_i))=\pi(\alpha(p_{i+1}))=\pi(v_0)$, $\pi(\alpha[p_i,p_{i+1}))$ equals either $\{\pi(v_0)\}$ or $\gamma$. Injectivity of $\pi$ over $\alpha[p_i,p_{i+1})$ rules out the former possibility.
\end{proof}

\begin{proposition}
  Let $P_0$ be as defined in \eqref{eq:def-P0}. Then $P_0= \{0\}$.  \label{prop:oneloop}
\end{proposition}
\begin{proof}
  We prove the proposition by showing that $m>1$ yields a
  contradiction.  Suppose that $m>1$. For each $0\le i<m$, let
  $w_i:=\alpha(p_i), \gamma_h^i := \alpha[p_i,p_{i+1})$ and define
  $\Psi_i\colon [0,1)\rightarrow \mathbb R$ as ${\displaystyle{\Psi_i
      := \phi\circ\left(\pi\big|_{\gamma_h^i}\right)^{-1}\circ
      \beta}}$.
  Note that $\Psi_i$ is well defined  for each $0\le i<m$  because
  $\pi:\gamma_h^i\rightarrow \gamma$ is a bijection from
  Proposition \ref{prop:pv}. Since it follows from Corollary
  \ref{cor:edge-homeo} that $\pi^{-1}:\gamma\rightarrow \gamma_h^i$ is
  continuous, we get that $\Psi_i$ is continuous for each $0\le i<m$.

  For convenience, denote $w_m=v_0=w_0$. By definition of $P_0$,
  $\pi(w_i)=\pi(v_0)$ for each $0\leq i\leq m$. From this and Theorem
  \ref{thm:nbd}, we have
  \begin{align}
    w_i &= \pi(w_i) + \phi(w_i)\,\hat{N}(\pi(w_i)) = \pi(v_0) +
    \phi(w_i)\,\hat{N}(\pi(v_0)). \label{eq:s4}
  \end{align}
  Since $w_i=v_0$ only for $i=0,m$, \eqref{eq:s4} implies that
  $\phi(w_i)\neq \phi(v_0)$ for any $1<i<m$. In particular, since
  $\phi(w_1)\neq \phi(v_0)$, without loss of generality, assume that
  $\phi(w_1)>\phi(v_0)$.  Then since $\phi(w_m)=\phi(v_0)$, there
  exists a smallest index $k$ such that (i) $1\leq k<m$, (ii)
  $\phi(w_k)\geq \phi(w_1)$ and (iii) $\phi(w_{k+1})<\phi(w_1)$.  For
  such a choice of $k$, consider the map
  $(\Psi_0-\Psi_k):[0,1)\rightarrow {\mathbb R}$. From
  $\phi(w_0)=\phi(v_0)$ and $\phi(w_k)\geq \phi(w_1)>\phi(v_0)$, we
  get
  \begin{align}
    (\Psi_0-\Psi_k)(0) = \phi(w_0)-\phi(w_k)<0. \label{eq:s5}
  \end{align}
  On the other hand, from $\phi(w_{k+1})<\phi(w_1)$, we get
  \begin{align}
    (\Psi_0-\Psi_k)(1^-) = \phi(w_1)-\phi(w_{k+1}) > 0. \label{eq:s6}
  \end{align}
  Eqs.\eqref{eq:s5}, \eqref{eq:s6} and the continuity of
  $\Psi_0-\Psi_k$ on $[0,1)$ imply that there exists $\xi\in (0,1)$
  such that $\Psi_0(\xi)=\Psi_k(\xi)$.  For this choice of $\xi$, let
  $x_0\in \gamma_h^0$ and $x_k\in \gamma_h^k$ be such that
  $\pi(x_0)=\pi(x_k)=\beta(\xi)$.  That $x_0$ and $x_k$ exist follows
  again, from Proposition \ref{prop:pv}. Now notice that
  $\Psi_0(\xi)=\Psi_k(\xi)\Rightarrow \phi(x_0)=\phi(x_k)$. Therefore
  from Theorem \ref{thm:nbd}, we have
  \begin{align}
    x_0=\pi(x_0) + \phi(x_0)\,\hat{N}(\pi(x_0)) = \pi(x_k) +
    \phi(x_k)\,\hat{N}(\pi(x_k)) = x_k. \label{eq:s9}
  \end{align}
  Eq.\eqref{eq:s9} shows that $\gamma_h^0\cap \gamma_h^k\neq
  \emptyset$. Since $\gamma_h$ is a simple curve (Lemma
  \ref{lem:simple-closed}) and $k\neq 0$, this is a contradiction.
\end{proof}

\begin{proof}[Proof of Lemma \ref{lem:homeo}]
  Propositions \ref{prop:pv} and \ref{prop:oneloop} together show that
  $\pi:\alpha([0,1))=\gamma_h\rightarrow \gamma$ is a bijection.
  Since $\pi$ is continuous on $\gamma_h$, it follows from Theorem
  \ref{thm:inv-cont} that $\pi:\gamma_h\rightarrow \gamma$ is a
  homeomorphism.
\end{proof}

\section{Connected components of $\Gamma_h$}
\label{sec:connected}
The final step in proving part (iv) of Theorem
\ref{thm:parameterization} is the following lemma.
\begin{lemma}
  Let $\gamma$ be a connected component of $\Gamma$, and  $\gamma_h :=
  \{x\in \Gamma_h: \pi(x)\in \gamma\}$. If $\gamma_h\not=\emptyset$, then
  $\gamma_h$ is a simple, closed curve, and a connected component of
  $\Gamma_h$. 
  \label{lem:connected}
\end{lemma}

To prove the lemma, it suffices to show that $\gamma_h$ is a connected
component of $\Gamma_h$, because then Lemma \ref{lem:simple-closed}
would imply that $\gamma_h$ is a simple, closed curve. To this end, we
consider the connected components $\{\gamma_h^i\}_{i=1}^m$ of
$\gamma_h$. Clearly $m<\infty$. The objective is to demonstrate that
$\gamma_h$ has just one connected component, i.e., that $m=1$. We do
so in simple steps. We first show in Proposition \ref{prop:components}
that each component $\gamma_h^i$ is in fact a connected component of
$\Gamma_h$ as well. Next, we order these connected components
according to their signed distance from $\gamma$ (Proposition
\ref{prop:ordering}). Then, we inspect the relative location of
triangles positively cut by each connected component with respect to
the rest. This reveals that $\gamma_h$ has just one connected
component.

\begin{proposition}
  For $i\in\{1,\ldots,m\}$, each connected component $\gamma_h^i$ of
  $\gamma_h$ is a connected component of $\Gamma_h$ as well, and
  consequently 
\begin{align}
  \pi:\gamma_h^i \rightarrow \gamma ~ \text{is a
    homeomorphism}. \label{eq:conn-homeo}
\end{align}
\label{prop:components}
\end{proposition}
\begin{proof}
  Clearly $\Gamma_h$ has only finitely many connected components, say
  $\{\Gamma_h^i\}_{i=1}^k$ for some $k<\infty$. We prove the
  proposition by demonstrating that for $i\in \{1,\ldots, m\}$ and
  $j\in \{1,\ldots, k\}$, $\gamma_h^i\cap \Gamma_h^j\neq
  \emptyset \Rightarrow \gamma_h^i=\Gamma_h^j$. 
  
  Suppose $\gamma_h^i\cap \Gamma_h^j\neq \emptyset$. Then
  $\pi(\gamma_h^i)\subseteq \pi(\gamma_h)\subseteq\gamma \Rightarrow
  \pi(\Gamma_h^j)\cap \gamma\neq \emptyset$. Using Lemma
  \ref{lem:homeo}, we get that $\pi:\Gamma_h^j\rightarrow \gamma$ is a
  homeomorphism, and in particular, $\pi(\Gamma_h^j)=\gamma$. By
  definition of $\gamma_h$, we get $\Gamma_h^j\subseteq
  \gamma_h$. Since $\gamma_h\subseteq \Gamma_h$, $\Gamma_h^j$ is a
  connected component of $\Gamma_h$ and $\Gamma_h^j\subseteq
  \gamma_h$, we conclude that $\Gamma_h^j$ is a connected component of
  $\gamma_h$ as well.  The assumption $\gamma_h^i\cap \Gamma_h^j\neq
  \emptyset$ implies that $\Gamma_h^j$ in fact equals
  $\gamma_h^i$. Eq. \eqref{eq:conn-homeo} follows immediately from
  Lemma \ref{lem:homeo}.
\end{proof}

Next, we order the connected components $\{\gamma_h^i\}_{i=1}^m$ of
$\gamma_h$ according to their signed distance from $\gamma$.  The
natural functions to consider for such an ordering are the maps
$\Psi_i=\phi\circ\left(\pi\big|_{\gamma_h^i}\right)^{-1}$, $1\leq i
\leq m$.

\begin{proposition}
  Let $1\le i,j\le m$. Then,
  \begin{romannum}
  \item The function $\Psi_i$ is well defined, continuous and for
    $K\in {\cal P}_h$ with positive edge $e\subset \gamma_h^i$, 
    \begin{align*}
      -h_K<\Psi_i(\pi(e))\leq h_K.
    \end{align*}
  \item For any $\xi\in \gamma$, $\Psi_i(\xi) = \Psi_j(\xi) ~ \iff ~ i=j$.
  \item If $\Psi_i(\xi)<\Psi_j(\xi)$ for some $\xi\in \gamma$, then
    $\Psi_i < \Psi_j$ on $\gamma$.
  \end{romannum}
  \label{prop:ordering}
\end{proposition}
\textit{Proof.}
\begin{romannum}
\item The fact that $\Psi_i$ is well-defined and continuous is a
  consequence of \eqref{eq:conn-homeo} and the continuity of
  $\phi$. Given positive edge $e\subset \gamma_h^i$ of $K\in
  {\cal P}_h$, part (ii) of Proposition \ref{prop:well-def} shows that
  $|\Psi_i(\pi(e)|\leq h_K$. That $\Psi_i(\pi(e))>-h_K$
  follows from \eqref{eq:phi-lower-hk}.
\item Let $\xi\in\gamma$ be arbitrary. Following
  \eqref{eq:conn-homeo}, let $x_i\in\gamma_h^i$ be such that
  $\pi(x_i)=\xi$, where $1\leq i \leq m$. From $\phi(x_i)=\Psi_i(\xi)$
  and Theorem \ref{thm:nbd}, we get
  \begin{align}
    x_i &= \pi(x_i) +\phi(x_i)\,\hat{N}(\pi(x_i)) = \xi + \Psi_i(\xi)\,\hat{N}(\xi). \label{eq:u1}
  \end{align}
  Since $\gamma_h^i\cap \gamma_h^j=\emptyset$ for $i\neq j$, $x_i=x_j\iff i=j$. Hence \eqref{eq:u1} implies that $\Psi_i(\xi)=\Psi_j(\xi)\iff i=j$. 
\item For some $i\neq j$ and $\xi\in \gamma$, assume that
  $\Psi_i(\xi)<\Psi_j(\xi)$.  Suppose there exists $\zeta\in \gamma$
  such that $\Psi_i(\zeta) \not< \Psi_j(\zeta)$. Since part (ii) shows
  $\Psi_i(\zeta)\neq \Psi_j(\zeta)$, we have $\Psi_i(\zeta)>
  \Psi_j(\zeta)$.  Note that $(\Psi_i-\Psi_j)$ is a continuous map on
  the connected set $\gamma$.  Therefore, from
  $(\Psi_i-\Psi_j)(\xi)<0$, $(\Psi_i-\Psi_j)(\zeta)>0$ and the
  intermediate value theorem, we know there exists $\zeta'\in \gamma$
  such that $(\Psi_i-\Psi_j)(\zeta')=0$. This contradicts part
  (ii). \qquad \endproof
\end{romannum}

The above proposition shows that we can find the connected component
$i^\sharp $ of $\gamma_h$ that is closest to $\gamma$ by simply
inspecting the values of $\Psi_j(\xi)$ for $1\leq j\leq m$ at any
$\xi\in\gamma$. Then, $\Psi_{i^\sharp}<\Psi_j$ on $\gamma$ for each
$j$ different from $i^\sharp$.

As noted previously, each set $\gamma_h^i$ is a Jordan curve. Hence
${\mathbb R}^2\setminus\gamma_h^i$ has precisely two connected
components, namely $\Omega_i^-$ and $\Omega_i^+$. The purpose of such
a decomposition of ${\mathbb R}^2$ is to examine the relative location
of the connected components of $\gamma_h$ and Proposition
\ref{prop:omegapm} shows how to pick them. To this end, we introduce
the curve $\omega$ defined as
\begin{subequations}
  \begin{align}
    \omega &=
    \{\xi-r_\omega(\xi)\,\hat{N}(\xi)\,:\,\xi\in\gamma\}, \label{eq:v0}\\
    \text{where}\quad
    r_\omega &= \frac{1}{2}\left(r_n-\Psi_{i^\sharp}\right). \label{eq:v0-b}
  \end{align}  
\end{subequations}
We will compare the distances of each connected component $\gamma_h^i$
of $\gamma_h$ from $\gamma$ to establish their relative locations. The
curve $\omega$ introduced above is useful in these calculations.

\begin{proposition}
  For $\xi\in \gamma$ and $1\leq i\leq m$, let $K\in {\cal P}_h$ be
  such that $\left(\pi\big|_{\gamma_h^i}\right)^{-1}(\xi)$ belongs to
  the positive edge of $K$. Then
  \begin{align}
    -r_n <\phi(\xi-r_{\omega}(\xi)\,\hat{N}(\xi))=-r_\omega(\xi)<\Psi_i(\xi).
    \label{eq:omega-main}
  \end{align}
  \label{prop:omega}
\end{proposition}
\textit{Proof.}  Following \eqref{eq:conn-homeo}, we know that there
is a unique point $x_i\in\gamma_h^i$ such that
$\pi(x_i)=\xi$. Therefore, we can find $K\in {\cal P}_h$ such that
$x_i$ belongs to the positive edge of $K$. From part(i) of Proposition
\ref{prop:ordering} and \eqref{eq:h-global}, we get that
$|\Psi_i(\xi)|\leq h_K<r_n$. The definition of $r_\omega$ then implies
$0<r_\omega(\xi)<r_n$. Hence Theorem \ref{thm:nbd} shows
$\phi(\xi-r_\omega(\xi)\,\hat{N}(\xi))=-r_\omega(\xi)$. The lower
bound in \eqref{eq:omega-main} follows.

  Next, from Proposition \ref{prop:ordering} and \eqref{eq:h-global},
  we have
  \begin{align}
    \Psi_j(\xi)>-h_K>-r_n \quad \text{for}~1\leq j\leq m. \label{eq:om-a}   
  \end{align}
  Using \eqref{eq:om-a} and the definition of $i^\sharp$, we get the
  upper bound in \eqref{eq:omega-main}:
  \begin{align*}
    r_\omega(\xi)+\Psi_i(\xi) \geq r_\omega(\xi)+\Psi_{i^\sharp}(\xi)
    =\frac{1}{2}\left(r_n+\Psi_{i^\sharp}(\xi)\right)>0~\Rightarrow~-r_\omega(\xi)<\Psi_i(\xi). \qquad \endproof
  \end{align*}


\begin{proposition}
  For each $1\leq i\leq m$, ${\mathbb R}^2\setminus\gamma_h^i$ has
  precisely two connected components $\Omega_i^-$ and $\Omega_i^+$,
  such that the non-empty set $\omega$ is contained in $\Omega_i^-$.
  \label{prop:omegapm}
\end{proposition}
\begin{proof}
  Firstly, note that $\omega$ is the image of $\gamma$ under a
  continuous map. Therefore, the assumption that $\gamma$ is connected
  implies that $\omega$ is a connected set.  Each connected component
  $\gamma_h^i$ is a simple, closed curve (Proposition
  \ref{prop:components} and Lemma \ref{lem:simple-closed}).  Therefore
  by the Jordan curve theorem, ${\mathbb R}^2\setminus \gamma_h^i$ has
  precisely two connected components. From Proposition
  \ref{prop:omega}, we know that $-r_\omega<\Psi_i$ on $\gamma$. Using
  this in the definition of $\omega$ implies that $\omega\cap
  \gamma_h^i=\emptyset$. Hence the connected set $\omega$ is contained
  in one of the two connected components of ${\mathbb R}^2\setminus
  \gamma_h^i$. The proposition follows from setting $\Omega_i^-$ to be
  the component of ${\mathbb R}^2\setminus\gamma_h^i$ that contains
  $\omega$ and $\Omega_i^+$ to be the other.
\end{proof}

\begin{proposition}
  For $\xi\in\gamma$ and $i\in \{1,\ldots,m\}$
  \begin{align}
    \emptyset \neq \ell_i^- := \left\{\xi+\lambda\,\hat{N}(\xi)\,:\,
      -r_n < \lambda < \Psi_i(\xi)\right\} \subset \Omega_i^-.
    \label{eq:lminus-main}
  \end{align}
  \label{prop:lminus}
\end{proposition}
\begin{proof}
  Following \eqref{eq:conn-homeo}, let $x_i\in\gamma_h^i$ be such that
  $\pi(x_i)=\xi$. From Theorem \ref{thm:nbd} and
  $\phi(x_i)=\Psi_i(\xi)$, we have
  \begin{align}
    x_i &= \xi + \phi(x_i)\hat{N}(\xi) = \xi+\Psi_i(\xi)\hat{N}(\xi). \label{eq:ww1}
  \end{align}
  Eq.\eqref{eq:ww1} demonstrates that $x_i\notin \ell_i^-$ and hence
  that $\ell_i^-\cap \gamma_h^i=\emptyset$. Then, noting that
  $\ell_i^-$ is a connected set, either $\ell_i^-\subset \Omega_i^-$
  or $\ell_i^-\subset \Omega_i^+$. Therefore, we prove
  $\ell_i^-\subset\Omega_i^-$ by showing that $\ell_i^-\cap
  \Omega_i^-\neq \emptyset$. To this end, consider the point $y =
  \xi-r_\omega(\xi)\hat{N}(\xi)$. While $y\in \omega$ by definition,
  $-r_n<-r_\omega(\xi)<\Psi_i(\xi)$ from Proposition \ref{prop:omega}
  shows that $y\in \ell_i^-$ (and hence $\ell_i^-\neq
  \emptyset$). Recalling that $\omega\subset \Omega_i^-$ from
  Proposition \ref{prop:omegapm} we get $y\in \ell_i\cap \omega\subset
  \ell_i\cap \Omega_i^-\Rightarrow \ell_i^-\cap \Omega_i^-\neq
  \emptyset$.
\end{proof}

\begin{proposition}
  For $\xi\in \gamma$ and $i\in\{1,\ldots,m\}$, 
  \begin{align}
    \emptyset\neq \ell_i^+ := \left\{\xi+\lambda\,\hat{N}(\xi)\,:\,
      \Psi_i(\xi)<\lambda <r_n \right\} \subset
      \Omega_i^+. \label{eq:lplus-main}
  \end{align}
  \label{prop:lplus}
\end{proposition}
\begin{proof}
  The set $\ell_i^+$ is non-empty because $\max_\gamma \Psi_i=
  \max_{\gamma_h^i}\phi <r_n$.  By definition, $\ell_i^+\cap
  \gamma_h^i=\emptyset$. Since $\ell_i^+$ is connected, it is either
  contained in $\Omega_i^-$ or in $\Omega_i^+$. Hence we prove the
  proposition by demonstrating that $\ell_i^+\cap\Omega_i^+\neq
  \emptyset$.

  Following Proposition \ref{prop:components}, let $x_i\in \gamma_h^i$
  be such that $\pi(x_i)=\xi$.
  Consider first the case in which $x_i$ is not a vertex in
  $\gamma_h^i$.  Let $e_{ab}$ be the edge in $\gamma_h^i$ that
  contains $x_i$. Since $\gamma_h^i$ is a Jordan curve, we know that
  there exists $\delta>0$ (possibly depending on $x_i$) such that
  ${B(x_i,\delta)}\cap \gamma_h^i$ is a connected set. Noting that
  $d(x_i,a),d(x_i,b)>0$ from $x_i\in{\bf ri}(\,e_{ab})$, and that
  $r_n\pm \Psi_i(\xi)>0$ from Proposition \ref{prop:ordering} and
  assumption \eqref{eq:h-global} choose $\varepsilon>0$ such that
  \begin{align}
    \varepsilon<\min\left\{\delta,d(x_i,a),d(x_i,b),r_n\pm\Psi_i(\xi)\right\}. \label{eq:w1}
  \end{align}
  In particular, $\varepsilon<\min\{\delta,d(x_i,a),d(x_i,b)\}$
  implies that ${B(x_i,\varepsilon)}\cap\gamma_h^i =
  {B(x_i,\varepsilon)}\cap e_{ab}$. Hence,
  ${B(x_i,\varepsilon)}\setminus\gamma_h^i$ has precisely two
  connected components $H_-$ and $H_+$, defined as $H_\pm = \left(
    B(x_i,\varepsilon)\setminus\gamma_h^i \right)\cap \Omega_i^\pm$.
  In particular, $H_-$ is a convex set (being the interior of a half
  disc).

  For the given point $\xi\in \gamma$, let $\ell_i^-$ be as defined in
  Proposition \ref{prop:lminus} and set $\zeta_\pm := x_i\pm
  (\varepsilon/2)\,\hat{N}(\xi)$.  From the definition of $x_i$ and
  $\zeta_\pm$, we get
  \begin{align}
    (\zeta_\pm-\xi)\cdot\hat{N}(\xi) = \Psi_i(\xi)\pm \frac{\varepsilon}{2}.
    \label{eq:w4}
  \end{align}
  From \eqref{eq:w4} and $0<\varepsilon<r_n-\Psi_i(\xi)$, we get
  $\zeta_+\in \ell_i^+\cap B(x_i,\varepsilon)$. Similarly,
  \eqref{eq:w4} and $0<\varepsilon<r_n+\Psi_i(\xi)$ show that
  $\zeta_-\in\ell_i^-\cap B(x_i,\varepsilon)$.  Using the latter and
  Proposition \ref{prop:lminus}, we get
  \begin{align}
    \ell_i^-\subset \Omega_i^- \Rightarrow \ell_i^-\cap
    B(x_i,\varepsilon)\subset H_- \Rightarrow \zeta_-\in
    H_-. \label{eq:w5}
  \end{align}
  Note that $\zeta_+\neq x_i \Rightarrow \zeta_+\notin
  \gamma_h^i$. Also, $\zeta_+\in H_-$ yields a contradiction because
  using $\zeta_-\in H_-$ and the convexity of $H_-$, we get
  \begin{align}
    \zeta_+\in H_-\Rightarrow \frac{1}{2}(\zeta_-+\zeta_+)\in H_-
    \Rightarrow \gamma_h^i\ni x_i \in H_- \Rightarrow \gamma_h^i\cap
    \Omega_i^-\neq \emptyset. \label{eq:w6}
  \end{align}
  Hence we get the required conclusion that
  \begin{align}
    \zeta_+\in H_+ \Rightarrow \ell_i^+\cap H_+\neq \emptyset
    \Rightarrow \ell_i^+\cap \Omega_i^+\neq \emptyset. \notag
  \end{align}

  The case in which $x_i$ is a vertex is similar. For brevity, we only
  provide a sketch of the proof and omit details. By Lemma
  \ref{cor:two-vertex}, precisely two positive edges in $\gamma_h^i$
  intersect at $x_i$. Let these edges be $e_{x_ia}$ and
  $e_{x_ib}$. Choose $\varepsilon$ as in \eqref{eq:w1} and define
  $H_\pm$ as done above.  Define $\zeta_\pm$ as above
  and note that $\zeta_-\in H_-$ as done in \eqref{eq:w5}. The main
  difference compared to the case when $x_i$ is not a vertex is that
  now, $H_-$ is either a convex or a concave set. If $H_-$ is convex,
  arguing as in \eqref{eq:w6} shows that $\zeta_+\in H_+$. To show
  $\zeta_+\in H_+$ when $H_-$ is concave, it is convenient to adopt a
  coordinate system. The essential step is noting that
  $\hat{T}(\xi)\cdot\hat{U}_{x_ia}$ and
  $\hat{T}(\xi)\cdot\hat{U}_{x_ib}$ have opposite (and non-zero) signs
  as shown by Lemma \ref{lem:simple}.
\end{proof}

\begin{corollary}
  Let $i,j\in\{1,\ldots,m\}$. If $\Psi_i(\zeta)<\Psi_j(\zeta)$ for
  some $\zeta\in \gamma$, then $\gamma_h^j\subset \Omega_i^+$.
  \label{cor:compinomega+}
\end{corollary}
\begin{proof}
  For an arbitrary point $x\in \gamma_h^j$, let $\xi=\pi(x)$ and
  define $\ell_i^+$ as in \eqref{eq:lplus-main}. Since
  $\Psi_i(\zeta)<\Psi_j(\zeta)$, Proposition \ref{prop:ordering} shows
  that $\Psi_i(\xi)<\Psi_j(\xi)$. From part (i) of the same
  proposition and \eqref{eq:h-global}, we also know that
  $\Psi_j<r_n$. Hence we get that $x\in \ell_i^+$. Since
  $\ell_i^+\subset \Omega_i^+$ (Proposition \ref{prop:lplus}), $x\in
  \ell_i^+\Rightarrow x\in \Omega_i^+$. Since $x\in \gamma_h^j$ was
  arbitrary, we conclude that $\gamma_h^j\subset \Omega_i^+$.
\end{proof}

\begin{proposition}
  Let $i,j\in \{1,\ldots,m\}$ and $K=(a,b,c)\in {\cal P}_h$ have
  positive edge $e_{ab}\subset \gamma_h^j$. If $\gamma_h^j\subset
  \Omega_i^+$, then $\overline{K}\subset \Omega_i^+$.
  \label{prop:componenttriangles}
\end{proposition}
\begin{proof}
  Note that $\gamma_h^j\subset \Omega_i^+$ immediately implies $i\neq
  j$.  Since $\gamma_h^i$ is a collection of positive edges, the set
  $K_h^i:=\overline{K}\cap \gamma_h^i$ is either empty, or a vertex of
  $K$ or an edge of $K$. From $i\neq j$, we get
  \begin{align}
    e_{ab}\cap \gamma_h^i \subseteq \gamma_h^j\cap\gamma_h^i=\emptyset. \label{eq:z1}
  \end{align}
  Therefore, neither $a$ nor $b$ belong to $K_h^i$. Hence $K_h^i$ does
  not contain any edge of $K$. Since every vertex in $\gamma_h^i$ has
  $\phi\geq 0$ but $\phi(c)<0$, $c\notin K_h^i$. Therefore we conclude
  that $K_h^i=\emptyset$.
  
  Since $\overline{K}$ is a connected set and
  $K_h^i=\overline{K}\cap\gamma_h^i =\emptyset$, either
  $\overline{K}\subset \Omega_i^+$ or $\overline{K}\subset
  \Omega_i^-$. However, $e_{ab}\subset\gamma_h^j\subset\Omega_i^+$
  shows that $\overline{K}\cap \Omega_i^+\neq \emptyset$. Hence
  $\overline{K}\subset\Omega_i^+$.
\end{proof}

\begin{proposition}
  Let $K=(a,b,c)\in {\cal T}_h$ and $e_{ab}\subset \gamma_h^i$. Then
  \begin{align}
    \overline{K} \cap \Omega_i^\pm\neq \emptyset
    ~\Rightarrow~ \overline{K}\setminus \gamma_h^i\subset
    \Omega_i^\pm.
    \label{eq:triside-a}
  \end{align}
  \label{prop:triside-a}
\end{proposition}
\begin{proof}
  It is convenient to consider the cases $\overline{K}\cap
  \Omega_i^-\neq \emptyset$ and $\overline{K}\cap \Omega_i^+\neq
  \emptyset$ simultaneously. Below we argue by contradiction to
  demonstrate that $\overline{K}\cap \Omega_i^\pm\neq \emptyset
  \Rightarrow \left(\overline{K}\setminus \gamma_h^i\right) \cap
  \Omega_i^\mp=\emptyset$. Then \eqref{eq:triside-a} follows from
  recalling that $\Omega_i^-,\Omega_i^+$ and $\gamma_h^i$ are pairwise
  disjoint, and that their union equals ${\mathbb R}^2$.
  
  To this end, let $x \in \overline{K}\cap \Omega_i^\pm$. Since
  $\gamma_h^i\cap \Omega_i^\pm=\emptyset$, $x\in
  \left(\overline{K}\setminus \gamma_h^i\right)\cap \Omega_i^\pm$.
  Suppose there exists $y\in \left(\overline{K}\setminus
    \gamma_h^i\right)\cap \Omega_i^\mp$. The assumptions $x\in
  \Omega_i^\pm$ and $y\in \Omega_i^\mp$ imply that line segment
  joining $x$ and $y$ necessarily intersects $\gamma_h^i$. Let point
  $z$ belong to this intersection. Since $\overline{K}\cap \gamma_h^i$
  is a union of one or more edges of $K$, $\overline{K}\setminus
  \gamma_h^i$ is a convex set. Therefore $x,y\in \overline{K}\setminus
  \gamma_h^i\Rightarrow z\in \overline{K}\setminus \gamma_h^i$, which
  contradicts the fact that $z\in \gamma_h^i$. This proves that
  $\overline{K}\cap \Omega_i^\pm\neq \emptyset \Rightarrow
  \left(\overline{K}\setminus \gamma_h^i\right)\cap
  \Omega_i^\mp=\emptyset$.
\end{proof}
\begin{remark}
  In the proposition above, $\overline{K}\setminus\gamma_h^i$ can be
  different from $\overline{K}\setminus e_{ab}$. Of course, if $K$ is
  positively cut, then neither $e_{ac}$ nor $e_{bc}$ can be positive
  edges and the proposition indeed states that $\overline{K}\cap
  \Omega_i^\pm\neq \emptyset \Rightarrow \overline{K}\setminus
  e_{ab}\subset \Omega_i^\pm$.  However, this need not be the case if
  $K$ is not positively cut. While $\gamma_h^i$ being simple
  (Proposition \ref{prop:components}) precludes the possibility of all
  three edges of $K$ being positive edges, it is possible that
  $K\notin {\cal P}_h$ has two positive edges. In such a case,
  $\overline{K}\setminus \gamma_h^i\subsetneq \overline{K}\setminus \overline e_{ab}$.
\end{remark}

\begin{proposition}
  Let $K_\pm=(a,b,c_\pm)\in {\cal T}_h$ and $1\leq i\leq m$. If
  $K_-\in {\cal P}_h$ and $e_{ab}\subset \gamma_h^i$, then
  $\overline{K}_\pm \setminus \gamma_h^i \subset \Omega_i^\pm$.
  \label{prop:triside-c}
\end{proposition}
\begin{proof}
  Let $x\in \textbf{ri}\,\left(e_{ab}\right)$,
  $\hat{t}:=\hat{T}(\pi(x))$, $\hat{n}:=\hat{N}(\pi(x))$ and
  $\ell:=\{\pi(x)+\lambda\,\hat{n}\,:\,-r_n<\lambda<r_n\}$. Since
  Proposition \ref{prop:angle-normal} shows $|\hat{n}\cdot
  \hat{U}_{ab}|<1$ and $h_{K_\pm}<r_n$, we know $\ell\cap K_\pm\neq
  \emptyset$. Hence pick $y_\pm \in \ell\,\cap K_\pm$ and note from
  Proposition \ref{prop:triside-b} that $\hat{U}_{xy_\pm}=\pm\hat{n}$.
  Consequently, $y_\pm\in \ell_i^\pm$ whence
 \begin{align}
   y_\pm\in \ell_i^\pm ~\Rightarrow~ K_\pm\cap \ell_i^\pm\neq
   \emptyset ~ \Rightarrow ~ K_\pm\cap \Omega_i^\pm\neq \emptyset
   ~\Rightarrow~ \overline{K}_\pm\setminus \gamma_h^i\subset
   \Omega_i^\pm,
   \label{eq:ts-7}
 \end{align}
 where we have used Propositions \ref{prop:lminus} and
 \ref{prop:lplus} for the penultimate, and Proposition
 \ref{prop:triside-a} for the last implication.
\end{proof}

\begin{proof}[Proof of Lemma \ref{lem:connected}]
  We need to show that $\gamma_h$ has only one connected component,
  i.e., that $m=1$.  We prove this by supposing that $m>1$ and
  arriving at a contradiction. Hence we suppose that $\gamma_h^1$ and
  $\gamma_h^2$ are connected components of $\gamma_h$. Proposition
  \ref{prop:ordering} shows that either $\Psi_1<\Psi_2$ or
  $\Psi_1>\Psi_2$ on $\gamma$. Without loss of generality, let us
  assume the former and note using Corollary \ref{cor:compinomega+}
  that $\gamma_h^2\subset \Omega_1^+$.
  
  Consider any positive edge $e_{uv}\subset \gamma_h^2$. By
  definition, we can find vertex $w$ such that triangle $K=(u,v,w)\in
  {\cal P}_h$ and $\phi(w)<0$. Since $\gamma_h^2\subset \Omega_1^+$,
  Proposition \ref{prop:componenttriangles} in particular shows that
  $w\in \Omega_1^+$. Below, we demonstrate that
  $w\in\Omega_1^+\Rightarrow \phi(w)\geq0$ to contradict the fact that
  $\phi(w)<0$.

  From \eqref{eq:h-global}, we know that $\xi:=\pi(w)$ is well
  defined. For this choice of $\xi$, let $\ell_1^\pm$ be as defined in
  \eqref{eq:lminus-main} and \eqref{eq:lplus-main}. From
  $|\phi(w)|<r_n$ (from Proposition \ref{prop:well-def} and
  \eqref{eq:h-global}) and $w=\xi+\phi(w)\,\hat{N}(\xi)$, we know
  $w\in\ell_1^-\cup\{x\}\cup\ell_1^+$. Clearly, $w\neq x$ because $x$
  is a point on a positive edge while $w$ is not. Since
  $w\in\Omega_1^+$, Proposition \ref{prop:lminus} shows that $w\notin
  \ell_1^-$. Since $\ell_1^-,\{x\},\ell_1^+$ are pairwise disjoint, we
  conclude that $w\in \ell_1^+$. Therefore
  $\phi(w)>\Psi_1(\xi)=\phi(x)$ and hence
  \begin{align}
    \phi(w)=\phi(x)+d(w,x). \label{eq:conn-1}
  \end{align}
  If $\phi(x)\geq 0$, \eqref{eq:conn-1} together with $x\neq w$ shows
  that $\phi(w)>0$ yielding the required contradiction.

  The case $\phi(x)<0$ remains.  In the following, we identify a point
  $y$ such that $\phi(w)>\phi(y)>0$ to arrive at the required
  contradiction.  To this end, following Proposition
  \ref{prop:components}, let $x\in \gamma_h^1$ be such that
  $\pi(x)=\xi$. Let $x$ belong to a positive edge $e_{ab}\subset
  \gamma_h^1$. Let $K_-=(a,b,c)\in {\cal P}_h$ be the triangle with
  positive edge $e_{ab}$; existence of $K_-$ follows from the
  definition of $e_{ab}$ being a positive edge and uniqueness follows
  from Lemma \ref{lem:uniq-id}.  Since $\phi(a)\geq 0$ and
  $\phi(x)<0$, continuity of $\phi$ on $e_{ab}$ shows that $\phi=0$ at
  some point in $e_{ab}$, i.e., $\exists\, z\in e_{ab}\cap
  \gamma$. Since $\gamma$ is immersed in ${\cal T}_h$, we can find a
  sufficiently small $\varepsilon>0$ such that
  $B(z,\varepsilon)\subset \text{int}(\omega_h)$, where $\omega_h$ is
  the polygonal domain triangulated by ${\cal T}_h$. In particular,
  the existence of such a ball shows that we can find triangle
  $K_+=(a,b,c_+)\in {\cal T}_h$ that has edge $e_{ab}$ in common with
  triangle $K_-$. From Lemma \ref{lem:uniq-id}, we know that
  $K_+\notin {\cal P}_h$ and hence that $\phi(c_+)>0$.

  Since $|\hat{N}(\xi)\cdot \hat{U}_{ab}|<1$ (Proposition
  \ref{prop:angle-normal}), the line
  $\ell=\{x+\lambda\,\hat{N}(\xi),~\lambda\in{\mathbb R}\}$
  necessarily intersects either $e_{ac_+}$ or $e_{bc_+}$. Without loss
  of generality, let us assume that $\ell$ intersects $e_{ac_+}$ at
  point $y$. Since $\pi$ is injective on $\gamma_h^1$ (Proposition
  \ref{prop:components}), $\pi(y)=\xi=\pi(x) \Rightarrow y\notin
  \gamma_h^1$. Since $y\in \overline{K}_+\setminus \gamma_h^1$ and
  Proposition \ref{prop:triside-c} shows $\overline{K}_+\setminus
  \gamma_h^1\subset \Omega_1^+$, we know $y\in \Omega_1^+$. Then,
  repeating the arguments used to show $w\in \ell_1^+$ and
  \eqref{eq:conn-1} also demonstrate that $y\in \ell_1^+$ and that
  \begin{align}
    \phi(y)= \phi(x)+d(x,y). \label{eq:conn-2}
  \end{align}
  By definition of $\vartheta_{K_-}^\text{adj}$ (see Def. 2.4(vi)),
  the interior angles in $K_+$ at vertices $a$ and $b$ are greater
  than or equal to $\vartheta_{K_-}^\text{adj}$. Therefore, we have
  $d(x,y)\geq d(a,x)\sin\vartheta_{K_-}^\text{adj}$. Using this and
  the lower bound for $\phi(x)$ from Corollary \ref{cor:phi-lower} in
  \eqref{eq:conn-2}, we get
  \begin{align}
    \phi(y) &\geq
    -2C_{K_-}^hd(a,x)d(a,b)+d(a,x)\sin\vartheta_{K_-}^\text{adj}, \notag\\
    &\geq
    d(a,x)\left(\sin\vartheta_{K_-}^\text{adj}-2C_{K_-}^hd(a,b)\right),\notag
    \\
    &\geq
    d(a,x)\left(\sin\vartheta_{K_-}^\text{adj}-2C_{K_-}^hh_{K_-}\right),
    \notag \\
    &>0. \qquad~ \left(\text{using $d(a,x)>0$ and
        \eqref{eq:Ch}}\right). \label{eq:conn-3}
  \end{align}
  
  Now, $x,y$ and $w$ are collinear points on the line segment
  $\overline{\ell}_1^+=\{\xi+\lambda\,\hat{N}(\xi), \Psi_1(\xi)\leq
  \lambda\leq r_n\}$ with $\lambda=\Psi_1(\xi),\phi(y)$ and $\phi(w)$
  respectively. Notice that vertex $w\notin \overline{K}_+$ because
  $\phi(w)<0$ while $\phi\geq 0$ at $a,b$ and $c_+$. Since
  $\{\xi+\lambda\,\hat{N}(\xi)\,:\,\Psi_1(x)\leq \lambda\leq
  \phi(y)\}\subset \overline{K}_+$, we conclude that $w\in
  \{\xi+\lambda\,\hat{N}(\xi)\,:\,\phi(y)<\lambda\leq r_n\}$ which in
  particular shows that $\phi(w)>\phi(y)$. In conjunction with
  \eqref{eq:conn-3}, we get that $\phi(w)>0$ yielding the required
  contradiction.
  
  In this way, we conclude that $m=1$, i.e.,
  $\gamma_h=\gamma_h^1$. Hence Proposition \ref{prop:components} shows
  that $\gamma_h$ is a connected component of $\Gamma_h$. In turn,
  Lemma \ref{lem:simple-closed} implies that $\gamma_h$ is a simple,
  closed curve.
\end{proof}

  



\section*{Proof of Theorem \ref{thm:parameterization}}
The theorem follows essentially from compiling results we have proved
thus far. 
\begin{romannum}
\item See Lemma \ref{lem:uniq-id}.
\item For a positive edge $e\subset \Gamma_h$, Lemma \ref{lem:jac}
  shows that $\pi$ is $C^1$ on $\textbf{ri}\,(e)$ with the Jacobian
  bounded away from zero. The inverse function theorem then implies
  that $\pi$ is a local $C^1$-diffeomorphism on
  $\textbf{ri}\,(e)$. Since $\pi$ is injective over
  $\textbf{ri}\,(e)$, the assertion follows.
\item See Corollary \ref{cor:phi-lower} for lower bound of $\phi$ and
  Proposition \ref{prop:well-def} for the upper bound. See Lemma
  \ref{lem:jac} for the bounds for the Jacobian.
\item With $m\geq 1$, let $\{\gamma^i\}_{i=1}^m$ be the distinct
  connected components of $\Gamma$. For each $i\in \{1,\ldots, m\}$,
  let $\gamma_h^i:=\{x\in \Gamma_h\,:\,\pi(x)\in \gamma^i\}$. By
  assumption, $\gamma_h^i\neq \emptyset$ for each $i$. It then follows
  from Lemma \ref{lem:connected} that  $\gamma_h^i$ is a simple,
  closed curve and a connected
  component of $\Gamma_h$, and from
  Lemma \ref{lem:homeo} that $\pi:\gamma_h^i\rightarrow \gamma^i$ is a
  homeomorphism, for each $i\in \{1,\ldots, m\}$.

\comment{AL: Details about continuity. One way of doing it would be to
  think about limits: let $\x_n\to x$, $x_n,x\in \Gamma_h$. Is then
  $\pi(x_n)\to \pi(x)$. Clearly this is the case if $x_n,x\in
  \gamma_h^i$ for some $n>n_0$, for some $i$. A condition that would
  guarantee this is that $d(\gamma_h^i,\gamma_h^j)>d_0>0$, for all
  $i\not=j$, namely, the curves are separated. This follows from the
  fact that each $\gamma_h^i$ is a closed set in $\mathbb R^2$, and
  since each one of them is disjoint from the other and $\mathbb R^2$
  is normal (see Munkres, Thm 2.3), then there are open disjoint
  subsets that contain each $\gamma_h^i$}

To show that $\pi:\Gamma_h\rightarrow \Gamma$ is a homeomorphism, it
is enough to show that it is continuous, one-to-one and onto (Theorem
\ref{thm:inv-cont}). Since $\cup_{i=1}^m\gamma^i=\Gamma$ and
$\cup_{i=1}^m\gamma_h^i=\Gamma_h$ by definition, it immediately
follows that $\pi:\Gamma_h\rightarrow \Gamma$ is continuous and
surjective.  It only remains to show that $\pi:\Gamma_h\rightarrow
\Gamma$ is injective.  Since we know from Lemma \ref{lem:homeo} that
$\pi$ is injective on each connected component of $\Gamma_h$, we only
need to consider the possibility that there exist
$j,k\in\{1,\ldots,m\}$ such that $j\neq k$ but $\pi(\gamma_h^j)\cap
\pi(\gamma_h^k)\neq \emptyset$.  Since
$\gamma^{j,k}=\pi(\gamma_h^{j,k})$, we have $\gamma^j\cap \gamma^k\neq
\emptyset$. Since $\gamma^j$ and $\gamma^k$ are connected components
of $\Gamma$, we in fact get $\gamma^j=\gamma^k$. Then Lemma
\ref{lem:connected} implies that the $\gamma_h^j\cup \gamma_h^k$ is a
connected set,
  which contradicts the fact that $\gamma_h^j$ and $\gamma_h^k$ are
  distinct connected components of $\Gamma_h$. 
  
\end{romannum}

\bibliographystyle{siam} 
\bibliography{references}

\appendix

\section{Distance and angle estimates} 
\label{sec:distances-angles} 
We prove Proposition \ref{prop:angle-normal}, the essential angle
estimate required in \S\ref{sec:edge-injective} to show injectivity of
$\pi$ over each positive edge and to bound its Jacobian. We begin with
a corollary of Proposition \ref{prop:gradpi}, that is useful when
estimating $\phi$ and $\nabla\phi$ in positively cut triangles while
knowing just their values at vertices of the
triangle.  \begin{corollary}[of Proposition \ref{prop:gradpi}] Let
  $K\in{\cal P}_h$ and $x,y\in \overline{K}$. Then,
  \begin{subequations}
    \label{eq:phi-gradphi-estimate}
    \begin{align}
      \left|\phi(y)-\left(y-\pi(x)\right)\cdot\hat{N}(\pi(x))\right|
      &\leq \frac{1}{2}C_K^hd(x,y)^2, \label{eq:phi-estimate} \\
      \text{and} \quad \left|\nabla\phi(y)-\nabla\phi(x)\right| &\leq
      C_K^hd(x,y). \label{eq:gradphi-estimate}
    \end{align}
  \end{subequations}
  \label{cor:phi-estimate}
\end{corollary}
\begin{proof}
  Let $L_{xy}\subset\overline{K}$ be the closed line segment joining
  $x$ and $y$. We have
  \begin{align}
    \max_{L_{xy}}\kappa\circ\pi \leq
    \max_{\overline{K}} \kappa\circ\pi \leq
    \max_{\overline{B(K,h_K)}\cap\Gamma}\kappa  = M_K \label{eq:c1}
  \end{align}
  and $|\phi|\leq h_K$ on $L_{xy}$. From $\sigma_KC_K^hh_K>0$ in
  \eqref{eq:sigmaCh}, it follows that $M_Kh_K<1$. Therefore, Proposition 
  \ref{prop:gradpi} implies the bound
  \begin{align}
    \left|\hat{U}_{xy}\cdot\nabla\nabla\phi(z)\cdot\hat{U}_{xy} \right|
    \leq \frac{\kappa(\pi(z))}{1-\left|\phi(z)\right|\kappa(\pi(z))}
    \leq \frac{M_K}{1-M_Kh_K}=C_K^h \quad \forall z\in L_{xy}. \label{eq:c2}
  \end{align}
  From Taylor's theorem, we have 
  \begin{subequations}
    \label{eq:c3}
    \begin{align}
      \left|\phi(y)-\phi(x)-\nabla\phi(x)\cdot (y-x)\right| &\leq 
      \frac{d(x,y)^2}{2}
      \max_{L_{xy}}\left|\hat{U}_{xy}\cdot\nabla\nabla\phi\cdot\hat{U}_{xy}\right|, \\
      \left|\nabla\phi(y)-\nabla\phi(x)\right| &\leq 
      d(x,y)
      \max_{L_{xy}}\left|\hat{U}_{xy}\cdot\nabla\nabla\phi\cdot\hat{U}_{xy}\right| 
    \end{align}
  \end{subequations}
  Using \eqref{eq:c2} and  $x=\pi(x)+\phi(x)\hat{N}(\pi(x))$ (Theorem \ref{thm:nbd})
  in \eqref{eq:c3} yields  \eqref{eq:phi-gradphi-estimate}.
\end{proof}

\begin{proposition}
  Let $K=(a,b,c)\in {\cal P}_h$ have positive edge $e_{ab}$. Then
  \begin{align}
    \hat{N}(\pi(x))\cdot \hat{U}_{yc} &\leq \cos\beta_K \quad \forall
    x,y\in e_{ab}. \label{eq:lemanglebd0-main}
  \end{align}
  \label{prop:angle-ac}
\end{proposition}
\begin{proof}
  Let $\hat{n}_x = \hat{N}(\pi(x))$. From Corollary
  \ref{cor:phi-estimate}, we have
  \begin{subequations}
    \label{eq:d12}
    \begin{align}
      \phi(i) &\leq \left(i-\pi(x)\right)\cdot \hat{n}_x +\frac{1}{2}C_K^hh_K^2  \quad \text{for}~i=a,b, \label{eq:d1} \\
      \text{and} 
      \quad \phi(c) &\geq \left(c-\pi(x)\right)\cdot \hat{n}_x-\frac{1}{2}C_K^hh_K^2. \label{eq:d2}
    \end{align}
  \end{subequations}
  By definition of $\eta_K$ in \eqref{eq:def-etaK}, we know
  \begin{align}
    \phi(i) - \phi(c) &\geq \eta_K h_K \quad \text{for}~i=a,b.
    \label{eq:d3}
  \end{align}
  Using \eqref{eq:d12} in \eqref{eq:d3}, we get
  \begin{align}
    \left(c-i\right)\cdot \hat{n}_x &\leq C_K^hh_K^2-\eta_Kh_K  \quad
    \text{for}~i=a,b. \label{eq:d4}
  \end{align}
  Since $y\in e_{ab}$, $y$ is a convex combination of $a$ and $b$,
  \eqref{eq:d4} implies that
  \begin{align}
    (c-y)\cdot \hat{n}_x &\leq C_K^hh_K^2-\eta_K h_K. \label{eq:d5}
  \end{align}
  Dividing \eqref{eq:d5} by $d(c,y)$ and noting that $\rho_K<d(c,y)\leq
  h_K$, we get
  \begin{align}
    \hat{U}_{yc}\cdot\hat{n}_x &\leq C_K^hh_K\frac{h_K}{\rho_K} - \eta_K\frac{h_K}{h_K}
    =  \sigma_KC_K^hh_K-\eta_K = \cos\beta_K,
  \end{align}
  which is the required inequality.
\end{proof}

\begin{proposition}
  Let $K=(a,b,c)\in {\cal P}_h$ have positive edge $e_{ab}$ and
  proximal vertex $a$. Then
  \begin{align}
    \hat{N}(\pi(a))\cdot \hat{U}_{ab} &\geq -\frac{1}{2}C_K^hh_K. \label{eq:anglebd1-main}
  \end{align}
  \label{lem:anglebd1}
\end{proposition}
\begin{proof}
  Since $a$ is the proximal vertex of $K$, $\phi(a)\leq
  \phi(b)$. Then, using Theorem \ref{thm:nbd}, we get 
  \begin{align}
    \phi(b) &\geq \phi(a) = \left(a-\pi(a)\right)\cdot
    \hat{N}(\pi(a)). \label{eq:e1}
  \end{align}
  From Corollary \ref{cor:phi-estimate}, we also have
  \begin{align}
    \phi(b) &\leq \left(b-\pi(a)\right)\cdot \hat{N}(\pi(a)) +
    \frac{1}{2}C_K^hd(a,b)^2. \label{eq:e2}
  \end{align}
  Comparing \eqref{eq:e1} and \eqref{eq:e2}, we get
  \begin{align}
    (b-a)\cdot\hat{N}(\pi(a)) &\geq -\frac{1}{2}C_K^hd(a,b)^2. \label{eq:e4}
  \end{align}
  Dividing \eqref{eq:e4} by $d(a,b)$ and using $d(a,b)\leq h_K$ yields
  \begin{align}
    \hat{U}_{ab}\cdot \hat{N}(\pi(a))  \geq -\frac{1}{2}C_K^hd(a,b) \geq -\frac{1}{2}C_K^hh_K, \label{eq:e5}
  \end{align}
  which is the required inequality.
\end{proof}

We can now prove Proposition \ref{prop:angle-normal}.\\
\noindent
\textit{Proof of Proposition \ref{prop:angle-normal}:}
  We first obtain the lower bound in \eqref{eq:angle-normal} by using
  the bound for $\hat{N}(\pi(a))\cdot\hat{U}_{ab}$ derived in
  Proposition \ref{lem:anglebd1}.  We have
  \begin{align}
    \hat{N}(\pi(x))\cdot\hat{U}_{ab}
    &= \hat{N}(\pi(a))\cdot \hat{U}_{ab} + \left(\hat{N}(\pi(x))-\hat{N}(\pi(a))\right)\cdot\hat{U}_{ab},  \notag \\
    &\geq -\frac{1}{2}C_K^hh_K
    -\left|\hat{N}(\pi(x))-\hat{N}(\pi(a))\right|,
    \qquad \left(\text{Proposition}~\ref{lem:anglebd1}\right)  \notag \\
    &= -\frac{1}{2}C_K^hh_K -
    \left|\nabla\phi(x)-\nabla\phi(a)\right|,
    \notag \\
    &\geq -\frac{1}{2}C_K^hh_K - C_K^hh_K. \qquad \qquad \qquad \qquad
    ~ ~ \left(\text{Corollary}~\ref{cor:phi-estimate}\right) \notag 
  \end{align}
  To derive the upper bound, we make use of the inequality
  \begin{align}
    \arccos(\hat{u}\cdot\hat{v}) 
    &\leq \arccos(\hat{u}\cdot\hat{w}) + \arccos(\hat{v}\cdot\hat{w}),
    \label{eq:trigid}
  \end{align}
  for any three unit vectors $\hat{u},\hat{v},\hat{w}$ in ${\mathbb
    R}^2$, with $\arccos\colon [-1,1]\to [0,\pi]$.  Setting $\hat{u}=\hat{N}(\pi(x))$, $\hat{v}=\hat{U}_{ac}$
  and $\hat{w}=\hat{U}_{ab}$ in \eqref{eq:trigid}, we get
  \begin{align}
    \arccos(\hat{N}(\pi(x))\cdot\hat{U}_{ab}) &\geq
    \arccos(\hat{N}(\pi(x))\cdot\hat{U}_{ac})
    -\arccos(\hat{U}_{ac}\cdot\hat{U}_{ab}). \label{eq:f2}
  \end{align}
  From Proposition \ref{prop:angle-ac}, we know
  $\hat{N}(\pi(x))\cdot\hat{U}_{ac}\leq \cos\beta_K$. Since
  $a$ is the proximal vertex in $K$, we have
  $\hat{U}_{ac}\cdot\hat{U}_{ab}=\cos\vartheta_K$. The upper bound in
  \eqref{eq:angle-normal} follows.

  Finally, to demonstrate that
  $\left|\hat{N}(\pi(x))\cdot\hat{U}_{ab}\right|<1$, it suffices to
  show that $\frac{3}{2}C_K^hh_K$ and $\cos(\beta_K-\vartheta_K)$ are
  both smaller than $1$. The latter follows from part (iv) of
  Proposition \ref{prop:well-def}. For the former, noting that
  $\sigma_K\geq \sqrt{3}$ in \eqref{eq:sigmaCh} yields $(3/2)C_K^hh_K
  \leq \sigma_KC_K^hh_K < \sin\vartheta_K/2 <1$. \qquad \endproof

  Part (ii) of Proposition \ref{prop:well-def} implies the lower bound
  $\phi\geq -h_K$ on the positive edge of $K\in {\cal P}_h$. This can
  be improved using the fact that $\phi\geq 0$ at each vertex in
  $\Gamma_h$.  The tighter bound computed below is used in
  \S\ref{sec:connected}.
\begin{corollary}[of Proposition \ref{lem:anglebd1}]
  Let $K=(a,b,c)\in {\cal P}_h$ have positive edge $e_{ab}$. Then
  \begin{align}
    \phi(x)\geq -2C_K^h\min\left\{d(a,x),d(b,x)\right\}d(a,b)
    \quad \forall x\in e_{ab}. \label{eq:phi-lower}
  \end{align}
  \label{cor:phi-lower}
\end{corollary}
\begin{proof}
 If $a$ is the proximal vertex of $K$, then \eqref{eq:e5} of
 the above proposition shows that 
  \begin{align}
    \hat{U}_{ab}\cdot\hat{N}(\pi(a))\geq- \frac{1}{2}C_K^hd(a,b). \label{eq:lbd-1}
  \end{align}
  Otherwise, $b$ is the proximal vertex of $K$ and we have 
  \begin{align}
  \hat{U}_{ab}\cdot \hat{N}(\pi(a)) 
  &=\hat{U}_{ab}\cdot
  \hat{N}(\pi(b))+\hat{U}_{ab}\cdot\left(\hat{N}(\pi(a))-\hat{N}(\pi(b))\right),
  \notag \\
  &\geq \hat{U}_{ab}\cdot
  \hat{N}(\pi(b)))-|\nabla\phi(a)-\nabla\phi(a)|, \notag \\
  &\geq \hat{U}_{ab}\cdot \hat{N}(\pi(b))-C_K^hd(a,b) \qquad
  \left(\text{from Corollary \ref{cor:phi-estimate}}\right), \notag \\
  &\geq -\frac{3}{2}C_K^hd(a,b). \qquad \qquad \qquad ~\quad
  \left(\text{using \eqref{eq:e5}}\right) \label{eq:lbd-2}
\end{align}
From \eqref{eq:lbd-1} and \eqref{eq:lbd-2}, we conclude that
\begin{align}
  \hat{U}_{ab}\cdot\hat{N}(\pi(a))\geq -\frac{3}{2}C_K^hd(a,b).\label{eq:lbd-3}
\end{align}
Next, using Corollary \ref{cor:phi-estimate}, we have
\begin{align}
  \phi(x)
  &\geq (x-\pi(a))\cdot\hat{N}(\pi(a)) -\frac{1}{2}C_K^hd(a,x)^2, \notag \\
  &= \phi(a) + (x-a)\cdot\hat{N}(\pi(a)) -\frac{1}{2}C_K^hd(a,x)^2,
  \quad \left(\pi(a)=a-\phi(a)\hat{N}(\pi(a))\right)
  \notag \\
  &\geq -d(a,x)\hat{U}_{ab}\cdot \hat{N}(\pi(a))
  -\frac{1}{2}C_K^hd(a,x)^2, \qquad ~\left(\phi(a)\geq 0\right) \notag\\
  &\geq -\frac{3}{2}C_K^hd(a,x)d(a,b)-\frac{1}{2}C_K^hd(a,x)^2, \qquad
  \quad ~  \left(\text{using}~\eqref{eq:lbd-3}\right)\notag\\
  &\geq -2C_K^hd(a,x)d(a,b). \qquad \qquad \qquad \qquad \qquad ~\left(\text{using}~d(a,x)<d(a,b)\right)
  \label{eq:lbd-4}
  \end{align}
  Of course, we can interchange the roles of $a$ and $b$ in the above
  calculations. The required lower bound for $\phi(x)$ follows.
\end{proof}

That the lower bound computed above is better than the trivial one
$\phi\geq -h_K$ is easily demonstrated.  Noting that $\sigma_K\geq
\sqrt{3}$ and $\vartheta_K<90^\circ$ (assumption
\eqref{eq:cond-angle}) in \eqref{eq:sigmaCh} yields
\begin{align}
  C_K^h h_K<\frac{1}{\sqrt{3}}\sin\frac{\vartheta_K}{2} \leq \frac{1}{\sqrt{6}}.
\end{align}
The estimate in \eqref{eq:phi-lower} then implies
\begin{align}
  \phi &\geq -C_K^hh_K^2 > -\frac{h_K}{\sqrt{6}}. \label{eq:phi-lower-hk}
\end{align}

\section{About the set of positive edges}
\label{sec:topology}
We prove Lemmas \ref{lem:uniq-id} and \ref{cor:two-vertex} here. We
proceed in simple steps, starting by examining the orientation of
positive edges with respect to the local normal and tangent to
$\Gamma$. From these calculations, we conclude that each edge in
$\Gamma_h$ is a positive edge of just one positively cut triangle
(Lemma \ref{lem:uniq-id}). This result in turn helps us show that at
least two positive edges intersect at each vertex in $\Gamma_h$ (Lemma
\ref{lem:closed}), a useful step in proving Lemma
\ref{cor:two-vertex}. In the following, $\text{sgn}:{\mathbb
  R}\rightarrow \{-1,0,1\}$ is the function defined as
$\text{sgn}(x)=x/|x|$ if $x\neq 0$ and $\text{sgn}(x)=0$ if $x=0$.

\begin{proposition}
  Let $(a,b,c)\in{\cal P}_h$ have positive edge $e_{ab}$ and proximal
  vertex $a$. Then
  \begin{subequations}
    \label{eq:sameside}
    \begin{align}
      {\rm sgn}(\hat{T}(\pi(a))\cdot\hat{U}_{ab}) &=
      {\rm sgn}(\hat{T}(\pi(a))\cdot\hat{U}_{ac})\neq0, \label{eq:sameside-a} \\
      \hat{N}(\pi(a))\cdot\hat{U}_{ac} &
      <\hat{N}(\pi(a))\cdot\hat{U}_{ab} . \label{eq:sameside-b}
    \end{align}
  \end{subequations}    
  \label{lem:sameside}
\end{proposition}
\begin{proof}
  For convenience, let $\hat{t}=\hat{T}(\pi(a))$ and
  $\hat{n}=\hat{N}(\pi(a))$. Let $\alpha_b,\alpha_c\in
  [0^\circ,360^\circ)$ denote the angles from $\hat{n}$ to
  $\hat{U}_{ab}$ and $\hat{U}_{ac}$ respectively measured in the
  clockwise sense so that
  \begin{align}
    \hat{U}_{ai} &= \cos\alpha_i\,\hat{n} + \sin\alpha_i\,\hat{t}
    \quad \text{for}~i=b,c. \label{eq:i1}
  \end{align}
  From \eqref{eq:i1} and the assumption that $a$ is the proximal
  vertex in $K$, note that
  \begin{align}
    \cos\vartheta_K = \hat{U}_{ab}\cdot\hat{U}_{ac} =
    \cos\alpha_b\cos\alpha_c+\sin\alpha_b\sin\alpha_c
    =\cos(\alpha_c-\alpha_b). \label{eq:i2}
  \end{align}

  First we prove \eqref{eq:sameside-a}. Since Proposition
  \ref{prop:angle-normal} shows $\hat{t}\cdot\hat{U}_{ab} \neq 0$,
  without loss of generality assume that $\hat{t}\cdot \hat{U}_{ab}>0
  \Rightarrow \alpha_b\in(0^\circ,180^\circ)$. The upper bound can be
  improved by invoking Proposition \ref{lem:anglebd1},
  \eqref{eq:sigmaCh} and $\sigma_K\geq \sqrt{3}$:
  \begin{align}
    \cos \alpha_b= \hat{n}\cdot\hat{U}_{ab} \geq -\frac{1}{2}C_K^hh_K \geq
    -\sigma_KC_K^hh_K > -\cos\vartheta_K \Rightarrow \alpha_b <
    180^\circ - \vartheta_K.
    \label{eq:i3}
  \end{align}
  Suppose then that $\hat{t}\cdot\hat{U}_{ac}\leq 0$, i.e.,
  $\alpha_c\geq 180^\circ$. From Propositions \ref{prop:well-def} and
  \ref{prop:angle-ac}, we have $\alpha_c\leq
  360^\circ-\beta_K<360^\circ-\vartheta_K$. In conjunction with
  \eqref{eq:i3}, this shows $\vartheta_K \le
  (\alpha_c-180^\circ)+\vartheta_K< \alpha_c-\alpha_b <
  360^\circ-\vartheta_K$ which clearly contradicts
  \eqref{eq:i2}. Therefore $\hat{t}\cdot\hat{U}_{ab}>0\Rightarrow
  \hat{t}\cdot\hat{U}_{ac}>0$ as well. The case
  $\hat{t}\cdot\hat{U}_{ab}<0$ is argued similarly.

  Next we show \eqref{eq:sameside-b}.  Following
  \eqref{eq:sameside-a}, without loss of generality assume that
  $\hat{t}\cdot\hat{U}_{ab}$ and $\hat{t}\cdot\hat{U}_{ac}$ are both
  positive. Consequently, $\alpha_b,\alpha_c\in
  (0^\circ,180^\circ)$. We proceed by contradiction. Suppose that
  $\hat{n}\cdot\hat{U}_{ab}\leq \hat{n}\cdot\hat{U}_{ac} \Rightarrow
  \alpha_c\leq \alpha_b$. Then, noting that
  $\cos\beta_K<\sigma_KC_K^hh_K$ (from \eqref{eq:def-betaK} and
  Proposition \ref{prop:well-def} part (iii)), $\cos\alpha_c\leq
  \cos\beta_K$ (Proposition \ref{prop:angle-ac}) and $\cos\alpha_b\geq
  -\sigma_KC_K^hh_K$ (Proposition \ref{lem:anglebd1}, $\sigma_K\ge
  \sqrt 3$), we get
  \begin{align}
    90^\circ-\arcsin(\sigma_KC_K^hh_K) < \beta_K \leq \alpha_c\leq
    \alpha_b \leq 90^\circ + \arcsin(\sigma_KC_K^hh_K), \label{eq:i4}
  \end{align}
  where $\arcsin\colon[-1,1]\to[-\pi/2,\pi/2]$. Together with
  \eqref{eq:sigmaCh}, this implies that
  $\alpha_b-\alpha_c <2\arcsin(\sigma_KC_K^hh_K)<2\times\vartheta_K/2 = \vartheta_K,$
  which contradicts \eqref{eq:i2}, and hence $\hat{n}\cdot\hat{U}_{ab}>
  \hat{n}\cdot\hat{U}_{ac}$. Again, the case in which both terms
  in \eqref{eq:sameside-a} are negative is handled similarly.
\end{proof}

\begin{proposition}
  Let $(a,b,c)\in{\cal P}_h$ have positive edge $e_{ab}$. Then 
  \begin{align}
    \text{\rm sgn}(\hat{T}(\pi(x))\cdot\hat{U}_{ab}) &=
    \text{\rm sgn}(\hat{U}_{ca}\cdot\hat{U}_{ab}^\perp)={\rm sgn}(\hat
    U_{cb} \cdot \hat U_{ab}^\perp)\quad \forall x\in
    e_{ab}.
    \label{eq:j-main}
  \end{align}
  \label{lem:top}
\end{proposition}
\begin{proof}
  Notice first that since
  \begin{equation}
    \label{eq:1}
    d(c,a) \hat U_{ca} = d(c,b)\hat U_{cb} + d(b,a) \hat U_{ba},
  \end{equation}
  it follows that ${\rm sgn}(\hat U_{ca}\cdot \hat U_{ab}^\perp) =
  {\rm sgn}(\hat U_{cb}\cdot \hat U_{ab}^\perp)$, after taking the
  inner product on both sides with $\hat U_{ab}^\perp$.  Without loss
  of generality then, assume that the proximal vertex in triangle
  $(a,b,c)$ is the vertex $a$.  For convenience, let $\alpha_i =
  \arccos(\hat{N}(\pi(a))\cdot\hat{U}_{ai})$ for $i=b,c$. From
  Proposition \ref{lem:sameside}, we know $\text{\rm
    sgn}(\hat{U}_{ab}\cdot\hat{T}(\pi(a))) = \text{\rm
    sgn}(\hat{T}(\pi(a))\cdot\hat{U}_{ac}) := \iota$.  From the
  definition of $\alpha_b,\alpha_c$ and $\iota$, we have
  \begin{subequations}
    \label{eq:j1}
    \begin{align}
      \hat{U}_{ai} &= \cos\alpha_i~\hat{n}+\iota\sin\alpha_i~\hat{t} \quad \text{for}~i=b,c, \label{eq:j1a} \\
      \hat{U}_{ab}^\perp&= \iota\sin\alpha_b~\hat{n}-\cos\alpha_b~\hat{t}, \label{eq:j1b}
    \end{align}
  \end{subequations}
  where we have again set $\hat t = \hat T(\pi(a))$ and $\hat n=\hat
  N(\pi(a))$. 
  Noting that $0^\circ<\alpha_b<180^\circ$ from Proposition
  \ref{prop:angle-normal} and $\alpha_b<\alpha_c$ from Proposition
  \ref{lem:sameside}, we get
  $0^\circ<\alpha_c-\alpha_b<180^\circ$. Then, using \eqref{eq:j1}, we
  have the following calculation:
  \begin{align}
    \text{sgn}(\hat{U}_{ca}\cdot\hat{U}_{ab}^\perp) 
    = \text{sgn}(\iota\sin(\alpha_c-\alpha_b))
    = \iota
    = \text{sgn}(\hat{t}\cdot\hat{U}_{ab}), \label{eq:j2}
  \end{align}
  which proves \eqref{eq:j-main} for $x=a$.  This in fact implies
  \eqref{eq:j-main} for every $x\in e_{ab}$. For if we suppose
  otherwise, then by continuity of the mapping
  $\hat{U}_{ab}\cdot\left(\hat{T}\circ\pi\right):e_{ab}\rightarrow
  {\mathbb R}$, there would exist $y\in e_{ab}$ such that
  $\hat{U}_{ab}\cdot\hat{T}(\pi(y))=0$, contradicting Proposition
  \ref{prop:angle-normal}.
\end{proof}

\begin{proof}[Proof of Lemma \ref{lem:uniq-id}]
  Let $e_{ab}$ be a positive edge in $\Gamma_h$.  By definition, we
  can find $K=(a,b,c)\in{\cal P}_h$ for which $e_{ab}$ is a positive
  edge.  Suppose that there exists $\tilde{K}=(a,b,d)\in {\cal P}_h$
  different from $K$ that also has positive edge $e_{ab}$. Then,
  applying Proposition \ref{lem:top} to triangles $K$ and $\tilde{K}$,
  we get
  \begin{align}
    \text{sgn}(\hat{U}_{ab}^\perp\cdot\hat{U}_{ca}) &=
    \text{sgn}(\hat{U}_{ab}^\perp\cdot\hat{U}_{da}),
    \label{eq:k1}
  \end{align}
  because both equal $\text{sgn}(\hat{U}_{ab}\cdot\hat{T}(\pi(a)))$.
  But \eqref{eq:k1} implies that $K\cap\tilde{K}\neq\emptyset$. This
  is a contradiction since $K$ and $\tilde{K}$ are non-overlapping
  open sets.  
\end{proof}

\begin{proposition}
  Let $K_\pm=(a,b,c_\pm)\in {\cal T}_h$ and $K_-\in {\cal P}_h$ have
  positive edge $e_{ab}$. If $x\in \textbf{ri}\,(e_{ab})$, then
  \begin{align}
    y\in \{\pi(x)+\lambda\,\hat{N}(\pi(x))\,:\,\lambda\in {\mathbb
      R}\}\cap K_\pm \, \Rightarrow \,
    \hat{U}_{xy}\cdot\hat{N}(\pi(x))=\pm 1.
  \end{align}
  \label{prop:triside-b}
\end{proposition}
\begin{proof}
  Denote $\hat{t}:=\hat{T}(\pi(x))$ and $\hat{n}:=\hat{N}(\pi(x))$. We
  consider first the case $y\in K_-$. 
  By choice of $y$, $x\neq y$ and hence $\hat{U}_{xy}$ is
  well-defined. Furthermore, $\hat{U}_{xy}$ is parallel to $\hat{n}$
  and hence
  \begin{align}
    \hat{U}_{xy}\cdot \hat{n} =
    \text{sgn}\left(\hat{U}_{xy}\cdot
      \hat{n}\right)\neq 0. \label{eq:ts-1}
  \end{align}
  From Proposition \ref{lem:top}, we know
  \begin{align}
    \text{sgn}\left(\hat{t}\cdot \hat{U}_{ab}\right) =
    -\text{sgn}\left(\hat{U}_{ab}^\perp\cdot \hat{U}_{ac}\right).
    \label{eq:ts-2}
  \end{align}
  However, $x\in e_{ab}$ and $y\in K_-$ implies 
  \begin{align}
    \text{sgn}\left(\hat{U}_{ab}^\perp\cdot \hat{U}_{xy}\right) =
    \text{sgn}\left(\hat{U}_{ab}^\perp\cdot \hat{U}_{ac}\right).
    \label{eq:ts-3}
  \end{align}
  Using \eqref{eq:ts-3} in \eqref{eq:ts-2} yields
  \begin{align}
    \text{sgn}\left(\hat{U}_{ab}^\perp\cdot \hat{U}_{xy}\right) =
    -\text{sgn}(\hat{t}\cdot\hat{U}_{ab}).
    \label{eq:ts-4}
  \end{align}
  
  Examining \eqref{eq:ts-4} above in a local coordinate system leads
  to the conclusion we seek. To this end, let
  $\alpha:=\arccos\left(\hat{n}\cdot \hat{U}_{ab}\right)$ and note
  from Proposition \ref{prop:angle-normal} that
  $0^\circ<\alpha<180^\circ$ and $\text{sgn}\left(\hat{t}\cdot
    \hat{U}_{ab}\right)\neq 0$.  We have
  \begin{subequations}
    \label{eq:ts-5}
    \begin{align}
    \hat{U}_{ab} &=\cos\alpha\,\hat{n}+\text{sgn}(\hat{t}\cdot
    \hat{U}_{ab})\sin\alpha\,\hat{t}.
    \label{eq:ts-5a}\\
    \hat{U}_{ab}^\perp &=\text{sgn}(\hat{t}\cdot
    \hat{U}_{ab})\sin\alpha\,\hat{n}-\cos\alpha\,\hat{t}.\label{eq:ts-5b}
  \end{align}
\end{subequations}
Evaluating \eqref{eq:ts-4} using \eqref{eq:ts-1} and \eqref{eq:ts-5}
yields
\begin{align}
  \text{sgn}\left( \left(\hat{n}\cdot\hat{U}_{xy}\right) \,
    \left(\hat{t}\cdot\hat{U}_{ab}\right)\, \sin\alpha\right) =
  -\text{sgn}\left(\left(\hat{t}\cdot\hat{U}_{ab}\right)\,
    \sin\alpha\right).
\label{eq:ts-6}
\end{align}
Noting that $\sin\alpha>0$ and $\hat{t}\cdot \hat{U}_{ab}\neq 0$ in
\eqref{eq:ts-6}, we conclude that
$\text{sgn}\left(\hat{n}\cdot\hat{U}_{xy}\right)=-1$, i.e.,
$\hat{U}_{xy}=-\hat{n}$.

Next, consider $y'\in \{\pi(x)+\lambda\,\hat{n}\,:\,\lambda\in
{\mathbb R}\}\in K_+$. Observe that since $K_-$ and $K_+$ are distinct
triangles sharing a common edge $e_{ab}$,
\begin{align}
  \text{sgn}\left(\hat{U}_{xy'}\cdot \hat{U}_{ab}^\perp\right) =
  -\text{sgn}\left(\hat{U}_{xy}\cdot \hat{U}_{ab}^\perp\right). \label{eq:ts-8}
\end{align}
Using $\hat{U}_{xy}=-\hat{n}$ and $\hat{n}\cdot \hat{U}_{ab}^\perp\neq
0$ (from \eqref{eq:ts-5b}) in \eqref{eq:ts-8} shows
$\hat{U}_{xy'}=\hat{n}$, which is the required result.
\end{proof}

\begin{proposition}
  Let $e_{pq}$ be an edge in ${\cal T}_h$ such that $\phi(p)\geq 0$
  and $\phi(q)<0$. Then $e_{pq}$ is an edge of two distinct triangles
  in ${\cal T}_h$.
  \label{prop:enum}
\end{proposition}
\begin{proof}
  Let $\omega_h$ be the domain triangulated by ${\cal T}_h$. To prove
  the lemma, it suffices to find a non-empty open ball centered at any
  point in $e_{pq}$ and contained in $\omega_h$.  To this end, note
  that since $\phi$ is continuous on $e_{pq}$ and has opposite signs
  at vertices $p$ and $q$, we can find $\xi \in \Gamma\cap
  e_{pq}$. Since $\Gamma$ is assumed to be immersed in ${\cal T}_h$,
  we know that $\Gamma\subset \text{int}(\omega_h)$. Therefore, there
  exists $\varepsilon>0$ such that $B(\xi,\varepsilon)\subset
  \text{int}(\omega_h)$, which is the required ball.
\end{proof}

The following lemma is the essential step in showing that connected
components of $\Gamma_h$ are closed curves.
\begin{lemma}
  At least two positive edges intersect at each vertex in $\Gamma_h$.
  \label{lem:closed}
\end{lemma}
\begin{proof}
  Let $a$ be any vertex in $\Gamma_h$. Since $\Gamma_h$ is the union
  of positive edges in ${\cal T}_h$, it follows that $a$ is a vertex
  of at least one positive edge. Suppose that $a$ is a vertex of just
  one positive edge, say $e_{ab_0}$. Then, we can find a triangle
  $(a,b_0,b_1)\in {\cal P}_h$ that has positive edge $e_{ab_0}$.
  Since $\phi(a)\geq 0$ and $\phi(b_1)<0$, applying Proposition
  \ref{prop:enum} to edge $e_{ab_1}$ shows that there exists
  $(a,b_1,b_2)\in {\cal T}_h$ different from $(a,b_0,b_1)$. Since
  $e_{ab_2}$ is not a positive edge, we know $\phi(b_2)<0$. Repeating
  this step, we find distinct vertices $b_1, b_2, \ldots b_n$ such
  that $(a,b_i,b_{(i+1)})\in {\cal T}_h$ for $i=0$ to
  $n-1$, $\phi(b_i)<0$ for $i=1$ to $n-1$ and terminate when $b_{n}$
  coincides with $b_0$.  That $n$ is finite follows from the
  assumption of finite number of vertices in ${\cal T}_h$. In
  particular, we have shown that $(a,b_0,b_1)$ and $(a,b_{n-1},b_0)$ are
  distinct triangles in ${\cal T}_h$ that are both positively cut by
  $\Gamma$ and have positive edge $e_{ab_0}$. This contradicts Lemma
  \ref{lem:uniq-id}.
\end{proof}

\begin{lemma}
  If $e_{ap}$ and $e_{aq}$ are distinct positive edges in ${\cal
    T}_h$, then
  \begin{align}
    \text{\rm sgn}(\hat{U}_{ap}\cdot\hat{T}(\pi(a))) &=
    -\text{\rm sgn}(\hat{U}_{aq}\cdot\hat{T}(\pi(a)) \neq 0.
    \label{eq:simple-main}
  \end{align}
  \label{lem:simple}
\end{lemma}

To prove the lemma, we will use the following corollary of Proposition
\ref{lem:top}. Note that unlike Proposition
\ref{lem:sameside}, $a$ need not be the proximal vertex in the result
below.
\begin{corollary}[of Proposition \ref{lem:top}]
  Let $(a,b,c)\in{\cal P}_h$ have positive edge $e_{ab}$ and denote
  $\hat{t} = \hat{T}(\pi(a))$ and $\hat{n}=\hat{N}(\pi(a))$.  Then
  \begin{align}
    \text{\rm sgn}(\hat{t}\cdot\hat{U}_{ab}) =
    \text{\rm sgn}(\hat{t}\cdot\hat{U}_{ac}) \Rightarrow
    \hat{n}\cdot\hat{U}_{ab} >
    \hat{n}\cdot\hat{U}_{ac}. \label{eq:angle-order-main}
  \end{align}
  \label{cor:angle-order}
\end{corollary}

\begin{proof}
  Let $\text{sgn}(\hat{t}\cdot\hat{U}_{ab}) =
  \text{sgn}(\hat{t}\cdot\hat{U}_{ac}) = \iota$ and
  $\alpha_i=\arccos(\hat{n}\cdot\hat{U}_{ai})$ for $i=b,c$. Using
  \begin{align}
    \hat{U}_{ca}\cdot\hat{U}_{ab}^\perp 
    = -\left(\cos\alpha_c\,\hat{n}+\iota\,\sin\alpha_c\,\hat{t} \right)\cdot
    \left(\iota\sin\alpha_b\,\hat{n}-\cos\alpha_b\,\hat{t} \right) = \iota\sin(\alpha_c-\alpha_b), \notag
  \end{align}
  and Proposition \ref{lem:top}, we get
  \begin{align}
    \iota 
    = \text{\rm sgn}(\hat{t}\cdot\hat{U}_{ab}) 
    = \text{\rm sgn}(\hat{U}_{ca}\cdot\hat{U}_{ab}^\perp)
    = \text{\rm sgn}(\iota\sin(\alpha_c-\alpha_b))
    = \iota\ \text{\rm sgn}(\sin(\alpha_c-\alpha_b)). \label{eq:m1}
  \end{align}
  Since $\iota\neq 0$ from Proposition \ref{prop:angle-normal}, and
  $\sin(\alpha_c-\alpha_b)\neq 0$ because edges $e_{ab}$ and $e_{ac}$
  in triangle $(a,b,c)$ cannot be parallel, we conclude that
  $\text{\rm sgn}(\sin(\alpha_c-\alpha_b))=1$. Hence
  $\alpha_c>\alpha_b$.
\end{proof}

\begin{proof}[Proof of Lemma \ref{lem:simple}]
  We proceed by contradiction.  Let $\hat{t}=\hat{T}(\pi(a))$ and
  $\hat{n}=\hat{N}(\pi(a))$. Proposition \ref{prop:angle-normal} shows
  that neither term in \eqref{eq:simple-main} equals zero.  Therefore,
  without loss of generality, suppose that
  \begin{align}
    \text{sgn}(\hat{t}\cdot\hat{U}_{ap})=\text{sgn}(\hat{t}\cdot\hat{U}_{aq})=1. \label{eq:n1}
  \end{align}
  Since $e_{ap}$ and $e_{aq}$ are distinct edges, \eqref{eq:n1}
  implies that $\hat{n}\cdot\hat{U}_{ap}\neq
  \hat{n}\cdot\hat{U}_{aq}$. Therefore, without loss of generality, we
  assume that
  \begin{align}
    \hat{n}\cdot \hat{U}_{ap} > \hat{n}\cdot\hat{U}_{aq}. \label{eq:n2}
  \end{align}
  Let $\{p_1,\ldots,p_n\}$ be a clockwise enumeration of all vertices
  in ${\cal T}_h$ such that $e_{ap_i}$ is an edge in ${\cal T}_h$ for
  each $i=1$ to $n$ and $p_1=p$. Let $m\leq n$ be such that $q=p_m$.
  Without loss of generality, we assume that $e_{ap_i}$ is not a
  positive edge for $i=2$ to $m-1$.  Denote by $\alpha_i\in
  [0^\circ,360^\circ)$, the angle between $\hat{n}$ and
  $\hat{U}_{ap_i}$ measured in the clockwise sense.  From
  \eqref{eq:n1} and \eqref{eq:n2}, we get that
  $0^\circ<\alpha_1<\alpha_m<180^\circ$.  Using the clockwise ordering
  of vertices, this implies that
  \begin{align}
    0^\circ<\alpha_1<\alpha_2<\ldots
    <\alpha_m<180^\circ. \label{eq:n3}
  \end{align}  
  
  Arguing by contradiction, we now show that $(a,p_1,p_2)\in{\cal
    T}_h$ and is positively cut.  Suppose that $(a,p_1,p_2)\notin
  {\cal P}_h$, which allows also for the possibility that
  $(a,p_1,p_2)\notin {\cal T}_h$ when $p_1$ and $p_2$ are not joined
  by an edge.  Then since $e_{ap_1}$ is a positive edge,
  $(a,p_n,p_1)\in {\cal T}_h$ and is positively cut. Note that the
  interior angle at $a$ in $(a,p_n,p_1)$, namely the angle between
  edges $e_{ap_n}$ and $e_{ap_1}$ measured in the clockwise sense, has
  to be smaller than $180^\circ$. Therefore, either
  $\alpha_n<\alpha_1$ or $\alpha_n-\alpha_1>180^\circ$. In either
  case, we have
  \begin{align}
    \hat{U}_{p_na}\cdot\hat{U}_{ap_1}^\perp 
    = -(\cos\alpha_n\,\hat{n}+\sin\alpha_n\,\hat{t})\cdot(\sin\alpha_1\,\hat{n}-\cos\alpha_1\,\hat{t})
    = \sin(\alpha_n-\alpha_1) 
    < 0. \label{eq:n4}
  \end{align}
  Using Proposition \ref{lem:top} in $(a,p_1,p_n)$, \eqref{eq:n1} and
  \eqref{eq:n4}, we get
  \begin{align}
    1 = \text{sgn}(\hat{U}_{ap_1}\cdot\hat{t}) = \text{sgn}(\hat{U}_{p_na}\cdot\hat{U}_{ap_1}^\perp) = -1, \notag
  \end{align}
  which is a contradiction. Hence, we conclude that $(a,p_1,p_2)\in
  {\cal T}_h$ and is positively cut. 

  Triangle $(a,p_1,p_2)$ being positively cut with positive edge
  $e_{ap_1}$ implies $\phi(p_2)<0$. Then Proposition \ref{prop:enum}
  shows that $(a,p_2,p_3)\in {\cal T}_h$. If $m\neq 3$, then
  $\phi(p_3)<0$ since $e_{ap_3}$ is not a positive edge. Repeating
  this step, we show that $(a,p_i,p_{(i+1)})\in {\cal T}_h$ for $i=1$
  to $m-1$ and that $\phi(p_i)<0$ for $i=2$ to $m-1$. In particular,
  we get that $(a,p_{m-1},p_m)\in {\cal T}_h$ and is positively
  cut. This contradicts Corollary \ref{cor:angle-order} because
  \eqref{eq:n3} shows that $\text{sgn}(\hat{t}\cdot\hat{U}_{ap_{m-1}})
  = \text{sgn}(\hat{t}\cdot\hat{U}_{ap_{m}})$ and
  $\hat{n}\cdot\hat{U}_{ap_{m-1}}>\hat{n}\cdot\hat{U}_{ap_m}$.\\
  An identical argument with an anti-clockwise ordering of vertices
  applies to the case when $\hat{t}\cdot\hat{U}_{ap}$ and
  $\hat{t}\cdot\hat{U}_{aq}$ are both strictly negative.
\end{proof}

\noindent
Lemma \ref{cor:two-vertex} follows immediately from Lemmas
\ref{lem:closed} and \ref{lem:simple}.

\end{document}